\documentclass[reqno, 11 pt]{amsart}%
\usepackage{amsthm}
\usepackage{amsmath}
\usepackage{amsfonts}
\usepackage{txfonts}
\usepackage{amssymb}
\usepackage{dsfont}
\usepackage{esint}
\usepackage{dsfont}
\usepackage{graphicx}
\usepackage[backref, colorlinks, linkcolor=red, anchorcolor=green, citecolor=blue]{hyperref}
\usepackage{color}
\usepackage{pdfrender,xcolor}
\usepackage{lmodern}
\usepackage{setspace}
\usepackage{bm}
\usepackage{url}
\usepackage[utf8]{inputenc}
\usepackage[T1]{fontenc}
\usepackage[english]{babel}
\usepackage{pdfrender,xcolor}
\usepackage{lmodern}
\usepackage{bm}
\usepackage{comment}
\usepackage[margin=1in]{geometry}
\usepackage{enumitem}
\usepackage{amsaddr}
\usepackage{empheq}%
\usepackage{mathtools}
\providecommand{\U}[1]{\protect\rule{.1in}{.1in}}
\numberwithin{equation}{section}

\newtheorem{theorem}{Theorem}

\newtheorem{lemma}{Lemma}
\newtheorem{proposition}{Proposition}

\newtheorem{definition}{Definition}
\newtheorem{remark}{Remark}

\allowdisplaybreaks

\newcommand\restr[2]{{% we make the whole thing an ordinary symbol
  \left.\kern-\nulldelimiterspace % automatically resize the bar with \right
  #1 % the function
  \vphantom{\big|} % pretend it's a little taller at normal size
  \right|_{#2} % this is the delimiter
  }}

\everymath{\displaystyle}

\begin{document}

\inputencoding{utf8}

\title[S\MakeLowercase{trong solutions for the} N\MakeLowercase{avier-}S\MakeLowercase{tokes-}V\MakeLowercase{oigt equations with non-negative density}]{S\MakeLowercase{trong solutions for the} N\MakeLowercase{avier-}S\MakeLowercase{tokes-}V\MakeLowercase{oigt equations with non-negative density}}

\author[H.B.~\MakeLowercase{de~}O\MakeLowercase{liveira}, K\MakeLowercase{h}.~K\MakeLowercase{hompysh}, A.G.~S\MakeLowercase{hakir}]{H.B.~\MakeLowercase{de} O\MakeLowercase{liveira}$^{1,2}$, K\MakeLowercase{h}.~K\MakeLowercase{hompysh}$^{3}$, A.G.~S\MakeLowercase{hakir}$^{3}$}

\address{$^{1}$FCT - Universidade do Algarve, Faro, Portugal\\
         $^{2}$CMAFcIO - Universidade de Lisboa, Lisbon, Portugal\\
         $^{3}$Al-Farabi Kazakh National University, Almaty, Kazakhstan}

\email{holivei@ualg.pt,  konat$_{-}$k@mail.ru, ajdossakir@gmail.com}

\selectlanguage{english}

\begin{abstract}

The aim of this work is to study the Navier-Stokes-Voigt equations that govern flows with non-negative density of incompressible fluids with elastic properties.
For the associated nonlinear initial-and boundary-value problem, we prove the global-in-time existence of strong solutions (velocity, density and pressure).
We also establish some other regularity properties of these solutions and find the conditions that guarantee the uniqueness of velocity and density.
The main novelty of this work is the hypothesis that, in some subdomain of space, there may be a vacuum at the initial moment, that is, the possibility of the initial density vanishing in some part of the space domain.

\bigskip

\bigskip

\end{abstract}

\date{\today}

\maketitle

\section{Introduction}\label{Sect:int}

In general terms, it can be said that the Navier-Stokes equations describe the evolution of the velocity field of an incompressible viscous fluid in the laminar regime.
It is one of the most important system of equations in mathematical physics and therefore has been widely studied by several authors during the last 120 years, either from the mathematical viewpoint or from the applications -- see \emph{e.g.}~\cite{CGR:2016,CF:1988,Galdi:2011,Joseph:1990,LS:1978}.
However the developed theory is still incomplete, specially in 3d, where general existence and uniqueness results for smooth solutions are still partial.
For this reason, many authors began to study slightly modified models of these equations for which it would be possible to prove the existence and uniqueness of smooth solutions.
Among these models are the so-called Navier-Stokes-Voigt equations for which it is possible to answer the still open questions of the Navier-Stokes equations.
In turn, the main feature of the Navier-Stokes-Voigt equations is that they can be used to model flows of viscous fluids with elastic properties, as for instance polymer solutions.
These materials are part of a wider class of fluids, called viscoelastic fluids, and can
exhibit all intermediate ranges of properties between an elastic solid and a viscous fluid. 
Typically, fluids that exhibit this behavior are macromolecular in nature and the most common examples are polymeric melts
and solutions used to make plastic articles, food products, such as dough used to make bread and pasta, and
biological fluids such as synovial fluids found in joints --  see \emph{e.g.}~\cite{CGR:2016,Joseph:1990}.
In most applications, density is considered a constant parameter. However, this assumption is unrealistic, because in almost all fluids the density varies, either with time or with the position of the space where the fluid element is. These variations may or may not be significant. There are situations in which the density cannot be assumed as a constant parameter as, for example, in the study of multi-phase flows consisting of several immiscible and incompressible fluids.

In this work, we are interested in studying density-dependent flows (nonhomogeneous flows) of incompressible fluids with elastic properties.
We assume the flow is governed by the following initial-and boundary-value problem,
\begin{align}
&
\operatorname{div}u =0 \qquad\mbox{in}\quad Q_T, \label{NSV:inc} \\
& (\rho u)_t + \operatorname{div}(\rho u\otimes u)
= \rho f -\nabla p+ \mu\Delta u +\kappa\,\Delta u_t\qquad\mbox{in}\quad Q_T, \label{NSV:mom} \\
&
\rho_t+\operatorname{div}(\rho u) =0,\quad \rho\geq 0 \qquad\mbox{in}\quad Q_T, \label{NSV:mass}\\
& \rho u =\rho_0u_0,\quad  \rho=\rho_0\qquad\mbox{in}\quad \{0\}\times\Omega, \label{NSV:ic} \\
&
u=0\qquad\mbox{on}\quad \Gamma_T. \label{NSV:u:bc}
\end{align}
Here, $Q_T$ denotes the time-space cylinder $(0,T)\times\Omega$, where $\Omega\subset\mathds{R}^3$ is a bounded domain and $T$ is a given positive constant. The lateral boundary $(0,T)\times\partial\Omega$ of $Q_T$ is denoted by $\Gamma_T$, and the other notation is defined as follows.
The unknowns of the problem are the velocity $u=(u_1,u_2,u_3)$, density $\rho$ and pressure $p$, while the external forces field $f$, initial velocity $u_0$ and initial density $\rho_0$ are given data.
Equations~\eqref{NSV:mom} and~\eqref{NSV:mass} are derived from the classical principles of conservation of momentum and mass, while \eqref{NSV:inc} expresses the incompressibility constraint of the fluid.
The system of equations \eqref{NSV:inc}-\eqref{NSV:mass} shall be denoted in the sequel as the Navier-Stokes-Voigt equations, where $\mu$ corresponds to the dynamic viscosity and $\kappa$  to the relaxation time, that is the characteristic time required for a viscoelastic fluid to relax from a
deformed state to its equilibrium configuration.
With this respect, we assume that they are both constant and such that
\begin{equation}\label{mu:k>0}
\mu>0,\quad \kappa>0.
\end{equation}
We assume the initial momentum, say $m_0$, is given by the product $\rho_0u_0$, where $u_0$ is the given initial velocity.
In this work, we are interested in the case of initial data $u_0$ and $\rho_0$ satisfying
\begin{align}
&
\label{cond:u0}
\operatorname{div}u_0=0\quad \mathrm{in}\quad \Omega, \\
&
\label{cond:r0}
0\leq \rho_0\leq M<\infty\quad \mathrm{in}\quad \Omega,
\end{align}
for some positive constant $M$.
The main novelty of this work lies in the assumption (\ref{cond:r0}), where it is understood that the initial density $\rho_0$ may eventually vanish in some domain $\omega\subset\subset\Omega$, i.e. the possibility that, at the initial moment, there might be vacuum in some part of the space domain.

We strength hypothesis (\ref{cond:u0}) with one of the following conditions,
\begin{alignat}{2}
&
\label{Hyp:u0:V}
u_0\in V, && \\
\label{Hyp:u0:V+H2}
&
u_0\in V\cap H^2(\Omega), &&
\end{alignat}
where $H^2(\Omega)$ is the usual $W^{2,2}(\Omega)$ Sobolev space and $V$ is the function space defined below at (\ref{space:Vs}).
For the definitions and notations of the function spaces used throughout the paper, we address the reader to the monographs~\cite{Galdi:2011,Mazya:2011}.
In particular, given $m\in\mathds{N}$ and $r\in[1,\infty]$, we denote by $L^r(\Omega)$ and $W^{m,r}(\Omega)$ the usual Lebesgue and Sobolev function spaces.
As usual, when $r=2$, we use the notation $H^m(\Omega)=W^{m,2}(\Omega)$.
By $W^{m,r}_0(\Omega)$, we denote the closure of $C^\infty_0(\Omega)$ in the norm of $W^{m,r}(\Omega)$. The dual space of $W^{m,r}_0(\Omega)$ is denoted by $W^{-m,r'}(\Omega)$, where $r'$ denotes the Hölder conjugate of $r$.

Condition (\ref{Hyp:u0:V}) is enough to obtain weak solutions, but to obtain solutions with further regularity condition (\ref{Hyp:u0:V+H2}) is mandatory
(see the authors works~\cite{AO:2022:RACSAM,AOKh:2021:Nonl,AOKh:2019:b}).
On the forcing term $f$, we shall also assume two distinct situations,
\begin{alignat}{2}
&
\label{Hyp:f:L2}
f\in L^2(0,T;L^2(\Omega)), && \\
&
\label{Hyp:f:Linfty}
f\in L^\infty(0,T;L^2(\Omega)). &&
\end{alignat}
Of course that $L^\infty(0,T;L^2(\Omega))\hookrightarrow L^2(0,T;L^2(\Omega))$, unless we consider the entire time interval $(0,\infty)$.
But, as we shall see, condition (\ref{Hyp:f:L2}) is sufficient for the results we aim to prove, being condition (\ref{Hyp:f:Linfty}) only required to get some further $L^\infty$-in time regularity.

The version of the problem (\ref{NSV:inc})-(\ref{NSV:ic}), with constant density, was intensively studied by Oskolkov in a series of works (see \cite{Oskolkov:1995} and references cited there in) who coined the name Kelvin-Voigt for the associated system of equations.
However, as observed by Zvyagin and Turbin~\cite{zvy-2010}, neither Kelvin nor Voigt have suggested any stress-strain relation, or system of governing equations, for viscoelastic fluids.
Currently, the Navier-Stokes-Voigt name for the associated system of equations seems to be the most accepted by the people working in this field, especially because this model is, in fact, an extension of the system of equations proposed by Voigt (for elastic materials that exhibit relaxation time) to model materials with viscoelastic properties.
Mathematically speaking, the interesting feature of this system of equations, as noted first by Ladyzhenskaya~\cite{Ladyzhenskaya:1966}, is that the relaxation term $\kappa\,\Delta u_t$ works as a regularization of the Navier-Stokes equations so that the corresponding problem has a unique global solution.
Since then, the same problem, or some of its variants, have been studied by many authors, in many settings and under different conditions, with respect to the existence, uniqueness and asymptotic behaviour of the solutions. See for instance \cite{AOKh:2021:Nonl,AOKh:2019:c,zvy-2010} and the references cited therein.
On the other hand, the works by Titi and his collaborators~\cite{Titi:2010a,Titi:2010b} make a clear relation between the homogeneous Navier-Stokes-Voigt equations and the turbulence modeling, in particular with Bardina turbulence models.
The same relation was touched on by Lewandowski \emph{et al.}~\cite{Lewandowski:2006,Lewandowski:2018}.
Existence of weak and strong solutions of nonlinear problems governed by the Navier-Stokes-Voigt equations (\ref{NSV:inc})-(\ref{NSV:mom}) and for some of its generalisations, with p-Laplacian diffusion and damping terms, and for nonhomogeneous flows (non-constant density), were studied by the authors in \cite{AOKh:2021:Nonl,AOKh:2019:c} in the case of a strictly positive initial density.
For results on nonhomogeneous flows governed by the incompressible Navier-Stokes equations, which corresponds to take $\kappa=0$ in the momentum equation (\ref{NSV:mom}), we address the reader to the works by Antontsev \emph{et al.}~\cite{AKM:1990} and by Ladyzhenskaya and Solonnikov~\cite{LS:1978} in the case of a strictly positive initial density, and for Simon~\cite{Simon:1990}, Lions~\cite{Lions:1996} and Desjardins~\cite{Desjardins:1997} in the case of a initial density that vanishes in some part of the space domain.
In all these works, the authors were primarily interested in the global-in-time existence of weak solutions in bounded domains of $\mathds{R}^d$, $d=2,\ 3$, or in the whole space $\mathds{R}^d$, $d\geq 2$, and in the uniqueness of solution in the case of $d=2$.
Moreover, the initial data were considered so that $u_0\in H^1(\Omega)$ and $\rho_0\in L^\infty(\Omega)$.
The case of a strictly positive initial density has been worked out also by many other authors during the last 20 years with respect to existence of weak and strong solutions, uniqueness, asymptotic stability and blow-up -- see, for instance, ~\cite{HLL:2021,ZSC:2022} and the references cited therein.
We just want to point out the work by Paicu \emph{et al.} in~\cite{PZZ:2013}, where the authors have proved global-in-time existence and uniqueness in the whole space $\mathds{R}^d$, $d=2,\ 3$, for $u_0\in H^s(\mathds{R}^2)$, for $s>0$, or $u_0\in H^1(\mathds{R}^3)$.
When the density vanishes in some space subdomain, the momentum equation \eqref{NSV:mom} degenerates into an elliptic equation, which makes it difficult to achieve the existence of strong solutions.
To overcome this difficulty, Choe and Kim~\cite{Choe-Kim:2003:a, Choe-Kim:2003:b} estimated $\|\nabla u_t\|_{L^2(\Omega)}$ by requiring that the initial data should satisfy a compatibility condition, expressed by the following Stokes problem,
\begin{alignat}{2}
\label{Stokes:u0n:1}
& \operatorname{div}u_0 =0\quad \mbox{in}\ \ \Omega, && \\
\label{Stokes:u0n:2}
& -\mu\Delta u_0 =\sqrt{\rho_0}g - \nabla p_0\quad \mbox{in}\ \ \Omega, && \\
\label{Stokes:u0n:3}
& u_0 =0\quad \mbox{on}\ \ \partial\Omega, &&
\end{alignat}
for some $p_0\in H^1(\Omega)$ and $g\in L^2(\Omega)$.
This condition, combined with the assumption that $u_0\in V\cap H^2(\Omega)$ and $\rho_0\in L^\infty(\Omega)$, were of the utmost importance to prove the local-in-time existence of strong solutions in bounded domains of $\mathds{R}^3$.
More recently, without requiring the compatibility condition \eqref{Stokes:u0n:1}-\eqref{Stokes:u0n:3}, Lu \emph{et al.}~\cite{LuShiZhong:2018} have proved that, if the initial density decays not too slowly as $|x|\longrightarrow\infty$, then the 2d problem in the whole plane $\mathds{R}^2$ admits a unique global-in-time strong solution. The decay condition on the initial density was written in the following form
\begin{equation}\label{compat:2}
\rho_0\overline{x}^a\in L^1(\Omega)\cap H^1(\Omega)\cap W^{1,q}(\Omega),\qquad \overline{x}:=\sqrt{e+\left|x\right|^2}\log^2\left(e+\left|x\right|^2\right),
\end{equation}
for some $a>1$ and $q>2$ (see~\cite{LuShiZhong:2018}).
Around the same time, Li~\cite{Li:2017} has proved the existence of local-in-time strong solutions in bounded domains of $\mathds{R}^3$, assuming only that $u_0\in V$ and $\rho_0\in L^\infty(\Omega)\cap W^{1,\gamma}(\Omega)$ for any $\gamma>1$. Uniqueness was also proved in~\cite{Li:2017} but for all $\gamma\geq 2$.
More recently, Danchin and Mucha~\cite{DM:2019} improved the results of~\cite{Li:2017}, assuming only that $u_0\in V$ and $\rho_0\in L^\infty(\Omega ) $ (without requiring any condition, be it regularity, strict positivity or compatibility condition that the initial density satisfies). The problem was considered in a spatial domain that can either be a bounded domain of $\mathds{R}^d$, $d=2,\ 3$, with a $C^2$ boundary $\partial\Omega$, or the torus $\mathds{T}^d$, $d=2,\ 3$.
In these conditions, the authors~\cite{DM:2019} have proved global-in-time existence of strong solutions and their uniqueness in 2d, and also in 3d, in the last case only if $\|\nabla u_0\|_{L^2(\Omega)}$ is suitably small.
About 2 years ago, He \emph{et al.}~\cite{HLL:2021} have proved global-in-time existence of strong solutions and its exponential stability in unbounded domains of $\mathds{R}^3$. These authors have considered, in addition, the difficult situation of a density-dependent viscosity.
Last year, Zhang \emph{et al.}~\cite{ZSC:2022} have extended the results of \cite{DM:2019} requiring that the initial velocity can be in a larger function space: $u_0\in H^s_0(\Omega)$ for $s>0$.
Nonhomogeneous flows, with initial vacuum, governed by the incompressible Navier-Stokes-Voigt equations were firstly studied by the authors in \cite{AO:2022:RACSAM}, where it was proved the existence of weak solutions in the whole space $\mathds{R}^d$, $d=2,\ 3,\ 4$.
There it were also proved some properties regarding the large-time behavior of the solutions in special unbounded domains.
In the present, work we are interested in studying the existence and uniqueness of strong solutions for the Navier-Stokes-Voigt problem (\ref{NSV:inc})-(\ref{NSV:u:bc}) with non-constant density.
For the sake of mathematical generality, we shall assume throughout the rest of the work that the space dimension is $d\geq 2$, knowing in advance that there will be restrictions on the upper bound of $d$ depending on the results we shall obtain.
As we shall see in the sequel, the gain in regularity promoted by the presence of the relaxation term $\kappa\,\Delta u_t$ in the momentum equation shall allow us to prove the global-in-time existence, as well the uniqueness, of a strong solution without invoking any extra condition on the initial density.

We are interested in strong solutions to the problem (\ref{NSV:inc})-(\ref{NSV:u:bc}) in the sense of the following definition.

\begin{definition}\label{Def:ws}
Let $d \geq2$ and assume the conditions (\ref{mu:k>0}), (\ref{cond:r0}), (\ref{Hyp:u0:V+H2}) and (\ref{Hyp:f:L2}) are fulfilled.
If all the derivatives of $\rho$, $u$ and $p$ involved in (\ref{NSV:inc})-(\ref{NSV:mass}) are regular distributions and the equations (\ref{NSV:inc})-(\ref{NSV:mass}) hold almost everywhere in ${Q_T}$, and if still $\rho$ and $u$ satisfy the initial and boundary conditions (\ref{NSV:ic})-(\ref{NSV:u:bc}), then the triple $(\rho,u,p)$ is said to be a strong solution of the problem (\ref{NSV:inc})-(\ref{NSV:u:bc}).
\end{definition}

The main results of this work address the issue of existence of strong solutions for the problem (\ref{NSV:inc})-(\ref{NSV:u:bc}).
The first one is given by the following theorem and requires minimal assumptions on the regularity of the boundary domain.

\begin{theorem}\label{thm:e:strong:1}
Let $2\leq d \leq4$ and assume that $\Omega$ is a bounded domain with $\partial\Omega$ Lipschitz-continuous.
If the conditions (\ref{mu:k>0}), (\ref{cond:r0}), (\ref{Hyp:u0:V+H2}) and (\ref{Hyp:f:L2}) are fulfilled, then there exists, at least, a solution $(\rho,u,p)$
for the problem (\ref{NSV:inc})-(\ref{NSV:u:bc}) and such that:
\begin{enumerate}
[topsep=0pt, leftmargin=0pt,itemindent=*,label=(\arabic*)]
\item $0\leq \rho \leq M$ in ${Q_T}$, $\rho\in C([0,T];L^q(\Omega))$ for all $q\geq 1$ and $\rho_t\in L^2(0,T;W^{-1,2}(\Omega))$;
\item $u\in L^\infty(0,T;V)$ and $\sqrt{\rho}u\in L^\infty(0,T;L^2(\Omega))$;
\item $u_t\in L^2(0,T;V)$ and $\sqrt{\rho}u_t\in L^2(0,T;L^2(\Omega))$;
\item $p\in C_w([0,T);L^2(\Omega))$;
\end{enumerate}
If, instead of (\ref{Hyp:f:L2}), is fulfilled (\ref{Hyp:f:Linfty}), then:
\begin{enumerate}
[resume,topsep=0pt,leftmargin=0pt,itemindent=*,label=(\arabic*)]
\item $u_t\in L^\infty(0,T;V)$ and $\sqrt{\rho}u_t\in L^\infty(0,T;L^2(\Omega))$.
\end{enumerate}
\end{theorem}

It is noteworthy that, in the Navier-Stokes setting (without the relaxation term $\kappa\Delta u_t$ in the momentum equation (\ref{NSV:mom})) for nonhomogeneous flows,
to prove that $u_t\in L^2(0,T;V)$ or $u_t\in L^\infty(0,T;V)$, extra assumptions are needed.
In fact, to prove these results in \cite{Choe-Kim:2003:a}, it was assumed that
 \begin{alignat}{2}
\label{Hyp:ft:L2}
& f_t\in L^2(0,T;L^2(\Omega)), &&  \\
\label{Hyp:f:L2:H1}
& f\in L^2(0,T;H^1(\Omega)). &&
\end{alignat}
Moreover, and in addition to (\ref{Hyp:ft:L2})-(\ref{Hyp:f:L2:H1}), it was required the initial data $u_0$ and $\rho_0$ should satisfy the  compatibility condition \eqref{Stokes:u0n:1}-\eqref{Stokes:u0n:3}.
With respect to assumption (\ref{Hyp:ft:L2}), it should be noted that, by the Sobolev imbedding $W^{1,2}(0,T)\hookrightarrow L^\infty(0,T)$ (see~\cite[Theoreme~VIII.7]{Brezis:1983}), the assumption $f_t\in L^2(0,T;L^2(\Omega))$ (whence $f\in W^{1,2}(0,T;L^2(\Omega)))$ would imply
(\ref{Hyp:f:L2}).

This way, we can realize that for the nonhomogeneous Navier-Stokes-Voigt problem we can prove the same regularity results without assuming the extra conditions (\ref{Hyp:ft:L2})-(\ref{Hyp:f:L2:H1}) on the forcing term $f$, nor requiring the compatibility problem (\ref{Stokes:u0n:1})-(\ref{Stokes:u0n:3}), or condition (\ref{compat:2}).
More importantly, our proof is technically much more accessible and much less time consuming.
This is only possible due to the presence of the relaxation term $\kappa\Delta u_t$ in the momentum equation (\ref{NSV:mom}).
Besides (\ref{est:un:K2})-(\ref{est:un:K2'}), see the proofs of (\ref{est:uj:K2}) and (\ref{est:uj:K2'}) below.

To obtain more regularity in the solutions, not only a smoother boundary domain is needed, but also more restrictions on the space dimension.
This is the aim of the next theorem.

\begin{theorem}\label{thm:e:strong:2}
Let $2\leq d \leq3$ and assume that $\Omega$ is a bounded domain with $\partial\Omega$ supposed to be of class $C^2$.
If the conditions (\ref{mu:k>0}), (\ref{cond:r0}), (\ref{Hyp:u0:V+H2}) and (\ref{Hyp:f:L2}) are fulfilled, then there exists, at least, a solution $(\rho,u,p)$
for the problem (\ref{NSV:inc})-(\ref{NSV:u:bc}) and such that, in addition to (1)-(4) of Theorem~\ref{thm:e:strong:1}, we have:
\begin{enumerate}
[topsep=0pt, leftmargin=0pt,itemindent=*,label=(\arabic*)]
\item $D^2u\in L^\infty(0,T;L^2(\Omega))$ and $D^2u_t\in L^2(0,T;L^2(\Omega))$;
\item $\nabla p\in L^2(0,T;L^2(\Omega))$;
\end{enumerate}
If, instead of (\ref{Hyp:f:L2}), is fulfilled (\ref{Hyp:f:Linfty}), then, in addition to (5) of Theoreom~\ref{thm:e:strong:1}, we have:
\begin{enumerate}
[resume,topsep=0pt,leftmargin=0pt,itemindent=*,label=(\arabic*)]
\item $D^2u_t\in L^\infty(0,T;L^2(\Omega))$;
\item $\nabla p\in L^\infty(0,T;L^2(\Omega))$.
\end{enumerate}
\end{theorem}

The greatest gain in regularity, relatively to the nonhomogeneous Navier-Stokes equations (see again~\cite{Choe-Kim:2003:a,LuShiZhong:2018}), is observed in the regularity results $D^2u_t\in L^2(0,T;L^2(\Omega))$ and $D^2u_t\in L^\infty(0,T;L^2(\Omega))$ (see (\ref{est:uj:K4}) and (\ref{est:uj:K4'}) below), which cannot at all be achieved if we remove the relaxation term $\kappa\Delta u_t$ from the momentum equation
(\ref{NSV:mom}).

The proofs of Theorems~\ref{thm:e:strong:1} and~\ref{thm:e:strong:2} will appear as a consequence of what is done later on in Sections~\ref{Sect:EWS:AP}-\ref{Sect:lim:n}.

\begin{remark}\label{C-alpha:reg:en}
In addition to the regularity results $D^2u\in L^\infty(0,T;L^2(\Omega))$ and $D^2u_t\in L^\infty(0,T;L^2(\Omega))$, we also have
\begin{equation*}
u,\ \frac{\partial u}{\partial t}\in L^\infty(0,T;C^{0,\alpha}(\overline{\Omega}))\quad \mbox{whenever}\quad
0<\alpha\leq 2-\frac{d}{2}\quad\mbox{and}\quad 2\leq d\leq 3.
\end{equation*}
This can be shown by combining the Sobolev inequalities (\ref{Sob:ineq:Lap:D2}) and (\ref{Sob:ineq:infty:Lap}) with the regularity results aforementioned.
See also Remark~\ref{C-alpha:reg} below.
\end{remark}

As a complementary result to the existence of strong solutions $(\rho,u,p)$, we provide, in the following theorem, the conditions that allow us to prove the uniqueness of $\rho$ and $u$.

\begin{theorem}\label{thm:u:strong}
Let $(\hat{u},\hat{p},\hat{\rho})$ and $(\overline{u},\overline{p},\overline{\rho})$ be two solutions of the problem (\ref{NSV:inc})-(\ref{NSV:u:bc}) with the same data and in the conditions of Theorem~\ref{thm:e:strong:2}.
If, in addition,
\begin{equation}\label{hyp:gra:ro}
\rho_{0}\in W^{1,\infty}(\Omega)
\end{equation}
and
\begin{alignat}{2}
& \label{eq:th1}
u_{0}\in W^{2,r}(\Omega)\cap V, && \\
& \label{eq:th11}
f\in L^{2}\left(0,T;L^{r}(\Omega)\right), &&
\end{alignat}
for
\begin{equation}
\label{main:a:uniq}
d<r\leq2^\ast.
\end{equation}
then $\hat{\rho}=\overline{\rho}$ and $\hat{u}=\overline{u}$.
\end{theorem}
Here, $2^\ast$ denotes the Sobolev conjugate of $2$, and note that (\ref{main:a:uniq}) implies $2\leq d\leq 3$.

The rest of the paper is organised as follows.
In Section~\ref{Sect:AR}, we provide important auxiliary results that will be used throughout the following sections.
Problem (\ref{NSV:inc})-(\ref{NSV:u:bc}) is approximated by two cascade of problems, being Sections~\ref{Sect:EWS:AP}-\ref{Sect:lim:j} devoted to proving the existence of the Galerkin approximations for the second cascade of approximate problems.
Sections~\ref{Sect:apriori:j}-\ref{Sect:lim:j} also prove some results that make the solutions of the two cascade of approximate problems strong.
In Section~\ref{Sect:lim:n}, we prove the existence of strong solutions for the original problem (\ref{NSV:inc})-(\ref{NSV:u:bc}).
Further regularity results are proved in Section~\ref{Sect:unique}, where we also establish the uniqueness of the velocity and density under additional assumptions on the problem data.

\section{Auxiliary results}\label{Sect:AR}
In this section, we introduce important auxiliary results
that we will be used in the course of our work.
We recall the definition of the following function spaces,
\begin{align}
% \label{space:V}
&  \mathcal{V} :=\{u \in C^{\infty}_0(\Omega):\operatorname{div}u =0\}, \nonumber \\
% \label{space:H}
&  H:=\mbox{closure of $\mathcal{V}$ in the norm of $L^2(\Omega)$}, \nonumber \\
&  V_q:=\mbox{closure of $\mathcal{V}$ in the norm of $W^{1,q}(\Omega)$}. \label{space:Vs}
\end{align}
The particular case of $q=2$ in (\ref{space:Vs}) will be denoted only by $V$.
In the sequel, we shall denote the inclusion $X\subset Y$
of two Banach spaces $X$ and $Y$ with a continuous imbedding
$X\xrightarrow[]{}Y$ by $X\hookrightarrow Y$. If the imbedding
$X\xrightarrow[]{}Y$ is compact, we denote the inclusion $X\subset Y$
by $X\hookrightarrow\hookrightarrow Y$.
\begin{comment}
We say that the (bounded)
boundary $\partial\Omega$ belongs to the class $C^{m,\alpha}$, with
$0\leq\alpha \leq 1$, if each point
$x=(x_1,\dots,x_{d-1},x_d)\in\partial\Omega$ has a neighborhood $U$
such that the set $U\cap\Omega$ is represented by the inequality
$x_d<f(x_1,\dots,x_{d-1})$ in some Cartesian coordinate system and
for some function $f$ in the Hölder space $C^{m,\alpha}(\omega)$,
where $\omega$ is the projection of $\Omega$ onto
$\mathds{R}^{d-1}$.
In the sequel, we may replace the regularity assumptions $C^{0,1}$ or $C^{1,1}$ on the boundary $\partial\Omega$ by $C^1$ or $C^2$, respectively, which imply the first whenever $\Omega$ is a convex domain. % (cf.~Adams~\cite[Theorem~1.31]{Adams:1970})
% Proposition 1.1.13 on p. 11 (and from Proposition 1.1.11 on p. 10) Fiorenza
In particular, the set $C^{0,1}$ is formed by functions that are Lipschitz-continuous, a result that sometimes will allow us to assume that $\partial\Omega$ is Lipschitz-continuous.
Note also that, for convex domains, $C^{m,1}$ coincides with $W^{m,\infty}$, the Sobolev space formed by functions $f\in L^1$ whose weak derivatives $D^mf$ are essentially bounded.
\end{comment}

We start by recalling the Sobolev, Moser and Morrey inequalities and the continuous and compact imbeddings that come with them.
\begin{lemma}\label{lem:Sob:in:1}
Let $\Omega$ be a bounded domain in $\mathds{R}^d$ with a Lipschitz-continuous boundary $\partial\Omega$.
If $v\in W^{1,r}_0(\Omega)$, then
\begin{alignat}{2}
\label{Sob:ineq:gen}
& \left\|v\right\|_{L^{r^\ast}(\Omega)}\leq C(r,d)\left\|\nabla v\right\|_{L^{r}(\Omega)},\quad r^\ast=\frac{dr}{d-r},\quad 1\leq r<d, && \\
\label{Moser:ineq:gen}
& \left\|v\right\|_{L^{q}(\Omega)}\leq C(r,q,d)\left\|\nabla v\right\|_{W^{1,r}(\Omega)},\quad d\leq q<\infty,\quad r=d, &&  \\
\nonumber %\label{Morrey:ineq:gen}
& \left[v\right]_{C^{0,\alpha}(\overline{\Omega})}\leq C(r,d)\left\|\nabla v\right\|_{L^{r}(\Omega)},\quad 0<\alpha\leq1-\frac{d}{r},\quad r>d. &&
\end{alignat}
Moreover, $W^{1,r}(\Omega)\hookrightarrow L^q(\Omega)$ if $1\leq q\leq r^\ast$ and $1\leq r<d$,
$W^{1,r}(\Omega)\hookrightarrow L^q(\Omega)$ if $d\leq q<\infty$ and $r=d$, and
$W^{1,r}(\Omega)\hookrightarrow C^{0,\alpha}(\overline{\Omega})$ if $\alpha=1-\frac{r}{d}$ and $r>d$.
In addition, $W^{1,r}(\Omega)\hookrightarrow\hookrightarrow L^q(\Omega)$ if $1\leq q< r^\ast$ and $1\leq r<d$,
$W^{1,r}(\Omega)\hookrightarrow\hookrightarrow L^q(\Omega)$ if $1\leq q<\infty$ and $r=d$, and
$W^{1,r}(\Omega)\hookrightarrow\hookrightarrow C^{0,\alpha}(\overline{\Omega})$ if $0<\alpha\leq1-\frac{d}{r}$ and $r>d$.
\end{lemma}
\begin{proof}
For the proof, we address the reader to Maz'ya~\cite[Chapter~2]{Mazya:2011}.
\end{proof}

For the sake of simplifying the writing, in the sequel, we shall use the notation $r^\ast$ with the broadest meaning that $r^\ast=\frac{rp}{d-r}$ if $r<d$, $r^\ast$ is any real in the interval $[1,\infty)$ if $r=d$, or $r^\ast=\infty$ if $r>d$.

In the following lemma, we collect two important generalizations of the Aubin-Lions compactness lemma.
\begin{lemma}\label{lem:AL}
If $X$, $E$ and $Y$ are Banach spaces such that $X\hookrightarrow\hookrightarrow E\hookrightarrow Y$, then
\begin{alignat}{2}
\label{Aubin:1} & L^r(0,T;X)\cap \left\{v:v_t\in L^1(0,T;Y)\right\}\hookrightarrow\hookrightarrow L^r(0,T;E)\ \ \mbox{ if } \ \ 1\leq r\leq \infty, && \\
\label{Aubin:2} & L^\infty(0,T;X)\cap \left\{v:v_t\in L^q(0,T;Y)\right\}\hookrightarrow\hookrightarrow C([0,T];E)\ \ \mbox{ if } \ \ 1< q\leq \infty. &&
\end{alignat}
\end{lemma}

\begin{proof}
See Simon~\cite[Corollary~4]{Simon:1987}.
\end{proof}

The following variant of de~Rham's lemma is of the utmost importance
to recover the pressure, after the velocity field and the density are found.
In what follows, by
%$ W^{-1,\eta^{\prime}}(\Omega)$ we denote the dual space of
%$ W_{0}^{1,\eta}(\Omega)$, and by
$\left\langle\cdot,\cdot\right\rangle$ we denote the duality paring between
$ W^{-1,q^{\prime}}(\Omega)$ and
$ W_{0}^{1,q}(\Omega)$.

\begin{lemma}\label{lemm:BP}
Let $1<q<\infty$ and $\varphi^{\ast}\in W^{-1,q^{\prime}}(\Omega)$. If
\begin{equation*}%\label{Rham_lemm}
\left\langle \varphi^{\ast},\varphi\right\rangle=0 \quad\forall\ \varphi\in
V_{q},
\end{equation*}
then there exists a unique $p\in L^{q}(\Omega)$, with $\int_{\Omega}p\,dx=0$, such that
\begin{equation*}
\left\langle \varphi^{\ast},\varphi\right\rangle =
\int_{\Omega}p \operatorname{div}\varphi\,dx\quad\forall\,\varphi\in W_{0}^{1,q}(\Omega).
\end{equation*}
Moreover, there exists a positive constant $C$ such that
\begin{equation*}
\left\|p\right\|_{L^{q}(\Omega)}\leq C \left\|\varphi^{\ast}\right\|_{ W^{-1,q^{\prime}}(\Omega)}.
\end{equation*}
\end{lemma}

\begin{proof}
The proof combines the results of Bogovski\v{\i}~\cite{Bogovskii:1980} and
Pileckas~\cite{Pileckas:1980} (see also Theorems~III.3.1 and III.5.3 of Galdi~\cite{Galdi:2011}).
\end{proof}

Next, some other auxiliary results that shall be used in the sequel are collected.
We start by recalling some useful inequalities related with the Sobolev inequalities stated in Lemma~\ref{lem:Sob:in:1}.
Here, the notation $|D^mu|$, where $m\in\mathds{N}$, stands for
\begin{equation*}
|D^mu|:=\sum_{|\gamma|=m}\left|\frac{\partial^{|\gamma|}u}{\partial^{\gamma_1}_{x_1}\cdots\partial^{\gamma_d}_{x_d}}\right|,\quad |\gamma|=\gamma_1+\cdots+\gamma_d.
\end{equation*}
In particular, $|Du|^2=|\nabla u|^2$ and $|D^2u|^2=|\nabla u_{x_1}|^2+\cdots +|\nabla u_{x_d}|^2$.
We say that the (bounded) boundary $\partial\Omega$ belongs to the class $C^{m,\alpha}$, with $0\leq\alpha \leq 1$, if
each point $x=(x_1,\dots,x_{d-1},x_d)\in\partial\Omega$ has a neighborhood $U$ such that the set $U\cap\Omega$ is represented
by the inequality $x_d<f(x_1,\dots,x_{d-1})$ in some Cartesian coordinate system and for some function $f$ in the H\"{o}lder space $C^{m,\alpha}(\omega)$, where here $\omega$ is the projection of $\Omega$ onto $\mathds{R}^{d-1}$.

\begin{lemma}
Let $\Omega$ be a bounded domain in $\mathds{R}^d$ and assume that $r\geq 1$.
If the boundary $\partial\Omega$ is assumed to be of class $C^{0,1}$, then
\begin{alignat}{2}
& \left\|\nabla u\right\|_{L^{r^\ast}(\Omega)} \leq C(r,d)\left\|D^2 u\right\|_{L^r(\Omega)}\qquad \forall\ u\in W^{2,r}(\Omega)\cap W^{1,r}_0(\Omega),  &&
\label{Sob:ineq:2:D2} \\
& \frac{1}{C(r,d)}\left\|\Delta u\right\|_{L^r(\Omega)}\leq \left\|D^2u\right\|_{L^r(\Omega)} \leq C(r,d)\left\|\Delta u\right\|_{L^r(\Omega)}\qquad \forall\ u\in W^{2,r}(\Omega)\cap W^{1,r}_0(\Omega). && \label{Sob:ineq:Lap:D2}
\end{alignat}
If $\partial\Omega$ is assumed to be of class $C^{1,1}$, then
\begin{alignat}{2}
& \left\|\nabla u\right\|_{L^{r^\ast}(\Omega)} \leq C(d,r)\left\|\Delta u\right\|_{L^r(\Omega)}\qquad \forall\ u\in W^{2,r}(\Omega)\cap
W^{1,r}_0(\Omega), && \label{Sob:ineq:2:Lap} \\
& \left\|u\right\|_{C^{0,\alpha}(\overline{\Omega})} \leq C(r,d)\left\|\Delta u\right\|_{L^r(\Omega)}\qquad \forall\ u\in W^{2,r}(\Omega)\cap
W^{1,r}_0(\Omega),\qquad
0<\alpha\leq 1-\frac{d}{r^\ast},\quad
2r>d. && \label{Sob:ineq:infty:Lap}
\end{alignat}
\end{lemma}
\begin{proof}
For the proof we address the reader to \cite[Lemma~1]{AOKh:2021:Nonl} (see also Maz'ya~\cite[Chapters~1-2]{Mazya:2011}).
\end{proof}

As we will do throughout this work, the notation $C=C(d,r)$, used in the previous lemma, emphasizes the fact that the positive constants $C$ considered in (\ref{Sob:ineq:2:Lap})-(\ref{Sob:ineq:infty:Lap}), where they are supposed to be all distinct, depend on the parameters $d$ and $r$.

\begin{comment}
In the sequel, we may replace the regularity assumptions $C^{0,1}$ or $C^{1,1}$ on the boundary $\partial\Omega$ by $C^1$ or $C^2$, respectively, which imply the first whenever $\Omega$ is a convex domain. % (cf.~Adams~\cite[Theorem~1.31]{Adams:1970})
% Proposition 1.1.13 on p. 11 (and from Proposition 1.1.11 on p. 10) Fiorenza
In particular, the set $C^{0,1}$ is formed by functions that are Lipschitz-continuous, a result that sometimes will allow us to assume that $\partial\Omega$ is Lipschitz-continuous.
Note also that, for convex domains, $C^{m,1}$ coincides with $W^{m,\infty}$.
This is the reason why some authors use the assumption that $\partial\Omega$ is of class $W^{1,\infty}$, or $W^{2,\infty}$, instead (see \emph{e.g.} Simon~\cite[Lemma~15]{Simon:1990}). % see Evans p. 279
\end{comment}

In the next lemma, we recall some of the properties of the Stokes operator. The operator
\begin{equation}\label{Stokes:op}
\begin{split}
\mathbb{A}:H^2(\Omega)& \cap  V\ \longrightarrow\ H \\
 u  &\quad \longmapsto\  -\mu\mathbb{P}\big(\Delta u\big)
\end{split}
\end{equation}
where $\mathbb{P}:L^2(\Omega)\rightarrow H$ is the Leray projection, is called the Stokes operator.
This operator establishes a correspondence between the solutions $ u$ of the stationary Stokes problems
\begin{align}
&
\operatorname{div} u=0 \qquad\mbox{in}\quad \Omega, \label{Stokes:inc} \\
& -\mu\Delta u = f-\nabla p\qquad\mbox{in}\quad \Omega, \label{Stokes:mom} \\
&  u=0 \qquad\mbox{on}\quad \partial\Omega, \label{Stokes:bc}
\end{align}
and the corresponding external forces $f$. {
Due to the symmetry of the Leray projection, % see Sohr, p. 131 and Constantin and Foias~\cite[Proposition~4.1]{CF:1988}),
it can be proved that
\begin{equation}\label{prop:Leray}
\mu\int_\Omega\nabla u:\nabla\varphi\,dx=\int_\Omega \mathbb{A}(u)\cdot\varphi\,dx
\qquad \forall\  u\in W^{2,2}(\Omega)\cap V,\quad \forall\ \varphi\in V.
\end{equation}
}
The following result follows from the existence and regularity theory for elliptic operators.

\begin{lemma}\label{Lem:reg:St}
Let $\Omega$ be a bounded domain in $\mathds{R}^d$, with the boundary $\partial\Omega$ assumed to be of class $C^2$.
If $f\in L^r(\Omega)$, with $1<r<\infty$, then there exist unique $ u\in W^{2,r}(\Omega)$ and $ p\in W^{1,r}(\Omega)$, with $\int_\Omega  p(x)\,dx=0$, such that
$( u, p)$ verify the Stokes system (\ref{Stokes:inc})-(\ref{Stokes:mom}) a.e. in $\Omega$ and $ u$ satisfies (\ref{Stokes:bc}) in the trace sense.
Moreover, there exists a positive constant $C=C(\mu,r,\Omega)$ such that
\begin{equation}\label{Stokes:est:r}
\left\| u\right\|_{W^{2,r}(\Omega)} + \left\| p\right\|_{W^{1,r}(\Omega)} \leq
C\left\|f\right\|_{L^r(\Omega)}.
\end{equation}
\end{lemma}
\begin{proof}
We address the proof to Galdi~\cite[Theorem IV.6.1]{Galdi:2011}.
\end{proof}
%This result still holds if $\Omega=\mathds{R}^d$, and, in this case, the constant $C$ in \eqref{Stokes:est:r} depends only on $\mu$~, $r$ and $d$ %(see~\cite[Lemma~2.3.2]{Sohr:2001}).

From (\ref{Stokes:est:r}), we easily derive the following estimate,
\begin{equation}\label{Lap:est:r}
\left\|D^2 u\right\|_{L^r(\Omega)} + \left\|\nabla p\right\|_{L^r(\Omega)} \leq
C\left\|f\right\|_{L^r(\Omega)}.
\end{equation}
Using the correspondence between the Stokes operator (\ref{Stokes:op}) and the forces field $f$, and taking $r=2$ in (\ref{Lap:est:r}), we also obtain
\begin{equation}\label{Stokes:est}
\left\|D^2 u\right\|_{L^2(\Omega)} + \left\|\nabla p\right\|_{L^2(\Omega)} \leq
C\left\|\mathbb{A}(u)\right\|_{L^2(\Omega)}.
\end{equation}

\section{Existence of approximate solutions}\label{Sect:EWS:AP}

The proof of the existence of solutions for the problem (\ref{NSV:inc})-(\ref{NSV:u:bc}) shall follow from the existence of suitable Galerkin approximations.
Let us consider the following family of subsets of $\Omega$,
$$\Omega_n:=\left\{x\in\Omega : \operatorname{dist}(x,\partial\Omega)>\frac{1}{n},\qquad n\in\mathds{N}\right\},$$ %\ \wedge\ \partial\Omega_n\in C^2
and the usual Friedrichs mollifying kernel $\eta_n(x):=\frac{1}{n^d}\eta\left(\frac{x}{n}\right)$.
Recall that
$$\eta_n\in C^\infty(\mathds{R}^d),\quad
\operatorname{supp}\eta_n\subset B\left(0,n\right),\quad
\int_{\mathds{R}^d}\eta_n(x)\,dx=1.$$
We regularize the initial data $u_0$ and $\rho_0$ by considering its mollifying functions, say $u_{0,n}$ and $\rho_{0,n}$, defined by
\begin{align}
& \label{moll:u0}
u_{0,n}(x):=\left(\eta_n\star u_0\right)(x)=\int_\Omega \eta_n(x-y)u_0(y)\,dy %=\int_{B\left(0,\frac{1}{n}\right)} \eta_n(y)u_0(x-y)\,dy
,\quad x\in \Omega_n, \\
& \label{moll:rho0}
\rho_{0,n}(x):=\left(\eta_n\star \rho_0\right)(x)+\frac{1}{n}=\int_\Omega \eta_n(x-y)\rho_0(y)\,dy+\frac{1}{n} %=\int_{B\left(0,\frac{1}{n}\right)} \eta_n(y)\rho_0(x-y)\,dy+\frac{1}{n}
,\quad x\in \Omega_n.
\end{align}
It is well-known that for any $w\in L^1_{\operatorname{loc}}(\Omega)$, its mollifying function $w_n=\eta_n\star w$ satisfies
\begin{alignat}{2}
\label{moll:1}
& w_n\in C^{\infty}(\Omega_n),\qquad w_n \xrightarrow[n\to\infty]{} w\ \ \mbox{a.e. in}\ \Omega, && \\
\label{moll:2}
&
w_n \xrightarrow[n\to\infty]{} w\ \ \mbox{in}\ L^p(\Omega),\quad \mbox{whenever}\ w\in L^p(\Omega),\ \ 1\leq p<\infty. &&
\end{alignat}
As $\operatorname{supp}\eta_n\subset B\left(0,n\right)$, one has
\begin{equation*}%\label{moll:u1}
\operatorname{supp}u_{0,n}\subset\Omega_n+B\left(0,n\right).
\end{equation*}
Moreover, in view of (\ref{cond:r0}), (\ref{moll:u0})-(\ref{moll:rho0}) and (\ref{moll:1})-(\ref{moll:2}), we have
\begin{alignat}{2}
\label{bd:rho0:n}
& 0<\frac{1}{n}\leq \rho_{0,n}\leq M^\ast:=M+1<\infty\quad \mathrm{in}\ \ \Omega, && \\
\label{bd:u0:n}
& \left\|u_{0,n}\right\|_{L^2(\Omega)}\leq \left\|u _{0}\right\|_{L^2(\Omega)}. &&
\end{alignat}
Moreover, provided $u_0$ is sufficiently regular, we also have
\begin{equation}\label{bd:Du0:n}
\left\|\nabla u_{0,n}\right\|_{L^2(\Omega)}\leq \left\|\nabla u_0\right\|_{L^2(\Omega)},\qquad
\left\|D^2u_{0,n}\right\|_{L^2(\Omega)}\leq \left\|D^2u_0\right\|_{L^2(\Omega)}.
\end{equation}

Given a large enough, but arbitrary, $n\in\mathds{N}$, we consider the following initial-and boundary-value problem,
\begin{align}
&
\operatorname{div}u =0 \qquad\mbox{in}\quad {Q_T}, \label{NSV:inc:n} \\
& (\rho u)_t + \operatorname{div}(\rho u\otimes u)
= \rho f -\nabla p+ \mu\Delta u +\kappa\,\Delta u_t\qquad\mbox{in}\quad {Q_T}, \label{NSV:mom:n} \\
&
\rho_t+\operatorname{div}(\rho u) =0 \qquad\mbox{in}\quad {Q_T}, \label{NSV:mass:n} \\
& \rho u =\rho_{0,n}u_{0,n},\quad \rho=\rho_{0,n} \qquad\mbox{in}\quad \{0\}\times \Omega, \label{NSV:ic:n}  \\
& u=0 \qquad\mbox{on}\quad  \Gamma. \label{NSV:bc:n}
\end{align}

We construct a solution to the problem (\ref{NSV:inc:n})-(\ref{NSV:bc:n}) by using a semi-discrete Galerkin scheme.
Since the Stokes operator is injective, self-adjoint and has a compact inverse (see
\emph{e.g.}~\cite[Propositions~4.2-4]{CF:1988}), there exists an
increasing sequence of positive eigenvalues $\lambda_{i}$ and a
sequence of corresponding eigenfunctions
$\psi_{i}\in H^2(\Omega)\cap V$ such that
\begin{equation}\label{Stokes_bas}
\mathbb{A}(\psi_{i})=\lambda_i\psi_{i}.
\end{equation}
Moreover, the family $\left\{\psi_{i}\right\}_{i\in\mathds{N}}$
can be made orthogonal in $H$ and orthonormal in
$V$. Given $j\in\mathds{N}$, let us consider the
$j$-dimensional space $X^j$ spanned by the first $j$
eigenvalues given by (\ref{Stokes_bas}): $\psi_{1}$,
\dots, $\psi_{j}$.
For each $j\in\mathds{N}$, and proceeding as in the proof of \cite[Theorem~1]{AOKh:2021:Nonl} (see also the proof of \cite[Proposition~1]{AO:2022:RACSAM}), we can prove the existence of approximate solutions
\begin{align}
& \label{app:sol:uj}
u^j_n\in C^1([0,T);X^j),\quad  u^j_n(x,t)= \sum_{i=1}^jc_i^j(t)\psi_{i}(x),\quad\psi_{i}\in X^j, && \\
& \label{app:sol:rhoj}
\rho^j_n\in C^1([0,T);C^1(\overline{\Omega})) &&
\end{align}
to the following system of $j+1$ ordinary differential equations
\begin{equation}\label{eq:weak:u:j:nc}
\begin{split}
&  \int_{\Omega}\rho^j_n(t)\left[\frac{\partial u^j_n(t)}{\partial t}+\left(u^j_n(t)\cdot\nabla \right)u^j_n(t)\right]\cdot\psi_i\,dx +
   \kappa\int_{\Omega} \frac{\partial\nabla u^j_n(t)}{\partial t}:\nabla \psi_i\,dx\\
&  +
   \mu\int_{\Omega}\nabla u^j_n(t):\nabla \psi_i\,dx  =\int_{\Omega}\rho^j_n(t)f(t)\cdot\psi_i\,dx,
\qquad i=1, \dots,j,
\end{split}
\end{equation}
\begin{equation}\label{eq:mass:n:j}
\frac{\partial\rho^j_n}{\partial t}+u^j_n\cdot\nabla\rho^j_n=0,
\end{equation}
%where $\frac{\partial u^j_n}{\partial t}$ and $\frac{\partial\rho^j_n}{\partial t}$ denote the partial derivatives of $u^j_n$ and $\rho^j_n$ with respect %to the time derivative $t$.
System (\ref{eq:weak:u:j:nc})-(\ref{eq:mass:n:j}) is supplemented with the following initial conditions
\begin{equation}\label{ic:u:rho:j}
\rho^j_nu^j_n=\rho_{0,n}^ju_{0,n}^j,\quad
u^j_n=u^j_{0,n},\quad
\rho^j_n=\rho^j_{0,n}\quad\mbox{in}\ \{0\}\times \Omega,
\end{equation}
where $u^j_{0,n}=P^j(u_{0,n})$, with $P^j$ denoting the orthogonal projection $P^j:V\longrightarrow X^j$ so that
\begin{equation}\label{ic:u:rho:j:u0}
u^j_n(0,x)=\sum_{i=1}^j c_i^j(0)\psi_i(x),\quad c_i^j(0)=c_{i,0}^j:=(u_{0,n},\psi_i),\quad i\in\{1,\dots,j\},
\end{equation}
where $(\cdot,\cdot)$ denotes the $L^2-$scalar product.
Since the operator $P^j$ is uniformly continuous, we can assume that
\begin{equation}\label{strong:conv:ic}
u^j_{0,n}\xrightarrow[j\to\infty]{} u_{0,n}\quad\mbox{in}\ L^2(\Omega)\cap W^{1,2}(\Omega).
\end{equation}
About the approximate initial density $\rho^j_{0,n}$, we assume that
\begin{equation}\label{sc:rn0}
\rho^j_{0,n}\in C^1(\overline{\Omega}),\qquad \rho^j_{0,n}\xrightarrow[j\to\infty]{} \rho_{0,n} \  \mbox{ in }\  L^p(\Omega)\ \ \forall\ p\in[1,\infty).
\end{equation}
From (\ref{bd:rho0:n})-(\ref{bd:Du0:n}) and (\ref{strong:conv:ic})-(\ref{sc:rn0}), one readily has
\begin{alignat}{2}
\label{bd:rho0:n:j}
& \frac{1}{n}\leq \rho_{0,n}^j\leq M^\ast<\infty\quad \mathrm{in}\quad \Omega, && \\
\label{bd:un0:j}
& \left\|u_{0,n}^j\right\|_{L^2(\Omega)}\leq \left\|u _{0}\right\|_{L^2(\Omega)}, && \\
\label{bd:Du0:n:j}
&
\left\|\nabla u_{0,n}^j\right\|_{L^2(\Omega)}\leq \left\|\nabla u_0\right\|_{L^2(\Omega)},\qquad
\left\|D^2 u_{0,n}^j\right\|_{L^2(\Omega)}\leq \left\|D^2 u_0\right\|_{L^2(\Omega)}.
\end{alignat}

Observe that, due to the regularity of $u^j_n$ and by linearity and continuity, one can derive from (\ref{eq:weak:u:j:nc})
\begin{equation}\label{eq:weak:u:j:nc:s}
\begin{split}
& \int_{\Omega}\left[\rho^j_n(t)\frac{\partial u^j_n(t)}{\partial t}+\rho^j_n(t)\left(u^j_n(t)\cdot\nabla \right)u^j_n(t)-\mu\Delta u^j_n(t)-\kappa\frac{\partial\Delta u^j_n(t)}{\partial t}\right]
\cdot\psi\,dx = \\
& \int_{\Omega}\rho^j_n(t)f(t)\cdot\psi\,dx\qquad \forall\ \psi\in H^2(\Omega)\cap V
\end{split}
\end{equation}
in the distribution sense on $(0,T)$.
Moreover, by using Lemma~\ref{lemm:BP}, it can also be proved the  existence of a unique approximate pressure
\begin{equation}\label{exist:p:n:j}
p^j_n\in C_w([0,T);L^2(\Omega)),\qquad\mbox{with}\ \int_\Omega p^j_n(t)\,dx=0,
\end{equation}
so that
\begin{equation}\label{eq:weak:u:j:nc:s:p}
\begin{split}
& \int_{\Omega}\left[\rho^j_n(t)\frac{\partial u^j_n(t)}{\partial t}+\rho^j_n(t)\left(u^j_n(t)\cdot\nabla \right)u^j_n(t)-\mu\Delta u^j_n(t)-\kappa\frac{\partial\Delta u^j_n(t)}{\partial t}\right]
\cdot\psi\,dx - \int_{\Omega}\rho^j_n(t)f(t)\cdot\psi\,dx = \\
& \int_{\Omega}p^j(t)\operatorname{div}\psi\,dx \qquad \forall\ \psi\in H^2(\Omega)\cap W^{1,2}_0(\Omega)
\end{split}
\end{equation}
holds in the distribution sense on $(0,T)$ -- see~\cite[Theorem~2]{AOKh:2021:Nonl}.

\section{A priori estimates independent of $j$}\label{Sect:apriori:j}

In this section, we aim to obtain independent of $j$ estimates for the approximate solutions $u^j_n$, $\rho^j_n$ and $p^j_n$.
We list the obtained estimates in several propositions according to the conditions that are imposed on the forcing term $f$.
We highlight the dependence of these estimates on the viscous and relaxation parameters $\mu$ and $\kappa$ to perceive the importance of the presence of the viscous term
$\mu\Delta u$ and of the relaxation one $\kappa\Delta u_t$, in the momentum equation (\ref{NSV:mom}), for the results we achieve.
It is important to note that these estimates are also independent of $n$.

\begin{proposition}\label{prop:est:j:L2}
Let $u^j_n$, $\rho^j_n$ and $p^j_n$ be the approximate weak solutions of the problem (\ref{NSV:inc:n})-(\ref{NSV:bc:n}) that have been found in
(\ref{app:sol:uj}), (\ref{app:sol:rhoj}) and (\ref{exist:p:n:j}).
\begin{enumerate}
[leftmargin=0pt,itemindent=*,label=(\arabic*)]
\item If (\ref{cond:r0}) holds true, then
\begin{equation}\label{est:rho:j:Q}
0<\frac{1}{n}\leq\inf_{x\in\overline{\Omega}}\rho_{0,n}^j(x)\leq\rho^j_n(x,t)\leq\sup_{x\in\overline{\Omega}}\rho_{0,n}^j(x)\leq M^\ast<\infty
\qquad\forall\ (x,t)\in {Q_T}.
\end{equation}
\item If $2\leq d\leq 4$ and (\ref{cond:r0}) and (\ref{Hyp:u0:V}) are verified,
then there exists an independent of $j$ (and $n$) positive constant $K_1$ such that
\begin{equation}\label{est:uj:K1}
\sup_{t\in(0,T)}\left(\left\|\sqrt{\rho^j_n(t)}u^j_n(t)\right\|^2_{L^2(\Omega)}+\kappa\left\|\nabla u^j_n(t)\right\|^{2}_{L^2(\Omega)}\right)
+\mu\int_{0}^{T}\left\|\nabla u^j_n(t)\right\|^2_{L^2(\Omega)}dt \leq K_1
\end{equation}
\item If $2\leq d\leq 4$ and  (\ref{mu:k>0}), (\ref{cond:r0}), (\ref{Hyp:u0:V}) and (\ref{Hyp:f:L2}) hold, then there exists an independent of $j$ (and $n$) positive constant $K_2$ such that
\begin{equation}\label{est:uj:K2}
\int_0^T\left(\left\|\sqrt{\rho^j_n(t)}\frac{\partial u^j_n(t)}{\partial t}\right\|^2_{L^2(\Omega)}+\kappa\left\|\frac{\partial\nabla u^j_n(t)}{\partial t}\right\|^2_{L^2(\Omega)}\right)dt \leq  K_2.
\end{equation}
\item If $2\leq d\leq 3$ and  (\ref{mu:k>0}), (\ref{cond:r0}), (\ref{Hyp:u0:V+H2}) and (\ref{Hyp:f:L2}) are verified,
then there exists an independent of $j$ (and $n$) positive constant $K_3$ such that
\begin{equation}\label{est:uj:K3}
\kappa\sup_{t\in(0,T)}\left\|D^2u^j_n(t)\right\|^2_{L^2(\Omega)} +\mu\int_0^T\left\|D^2u^j_n(t)\right\|^2_{L^2(\Omega)}dt \leq K_3.
\end{equation}
\item If $2\leq d\leq 3$ and (\ref{mu:k>0}), (\ref{cond:r0}), (\ref{Hyp:u0:V+H2}) and (\ref{Hyp:f:L2}) hold, then there exists an independent of $j$ (and $n$) positive constant $K_4$ such that
\begin{equation}\label{est:uj:K4}
\kappa^2\int_0^T\left\|\frac{\partial D^2u^j_n(t)}{\partial t}\right\|_{L^2(\Omega)}^2dt
\leq K_4.
\end{equation}
\end{enumerate}
\end{proposition}

In the following proposition, we improve some of the estimates established in Proposition~\ref{prop:est:j:L2} by requiring that (\ref{Hyp:f:Linfty}) is fulfilled,
instead of (\ref{Hyp:f:L2}).

\begin{proposition}\label{prop:est:j:Linfty}
Let $u^j_n$, $\rho^j_n$ and $p^j_n$ be the approximate weak solutions of the problem (\ref{NSV:inc:n})-(\ref{NSV:bc:n}) that have been found in
(\ref{app:sol:uj}), (\ref{app:sol:rhoj}) and (\ref{exist:p:n:j}).
\begin{enumerate}
[leftmargin=0pt,itemindent=*,label=(\arabic*)]
\item If $2\leq d\leq 4$ and (\ref{mu:k>0}), (\ref{cond:r0}), (\ref{Hyp:u0:V}) and (\ref{Hyp:f:Linfty}) hold, then there exists an independent of $j$ (and $n$) positive constant $K_2'$ such that
\begin{equation}\label{est:uj:K2'}
\sup_{t\in[0,T]}\left(\left\|\sqrt{\rho^j_n(t)}\frac{\partial u^j_n(t)}{\partial t}\right\|^2_{L^2(\Omega)}+\kappa\left\|\frac{\partial\nabla u^j_n(t)}{\partial t}\right\|^2_{L^2(\Omega)}\right)\leq
K_2'.
\end{equation}
\item If $2\leq d\leq 3$ and (\ref{mu:k>0}), (\ref{cond:r0}), (\ref{Hyp:u0:V+H2}) and (\ref{Hyp:f:Linfty}) hold, then there exists an independent of $j$ (and $n$) positive constant $K_4'$ such that
\begin{equation}\label{est:uj:K4'}
\kappa^2\sup_{t\in[0,T]} \left\|\frac{\partial D^2u^j_n(t)}{\partial t}\right\|_{L^2(\Omega)}^2\leq K_4'.
\end{equation}
\end{enumerate}
\end{proposition}

Before we proceed to the proof of Propositions~\ref{prop:est:j:L2} and \ref{prop:est:j:Linfty}, let us make some comments on the estimates obtained here, in particular the relation between these estimates and its counterparts of the Navier-Stokes setting, i.e. when considering $\kappa=0$ in the momentum equation (\ref{NSV:mom}).

\begin{remark}\label{rem:est}
\rm
\begin{enumerate}
[leftmargin=0pt,itemindent=*,label=(\arabic*)]
\item
Estimate (\ref{est:uj:K1}) has already been established in ~\cite[Lemma~4-(2)]{AO:2022:RACSAM}, but here we just have required that
$f\in L^2(0,T;L^2(\Omega))$.
Moreover, in view of (\ref{est:uj:K2:0}) below (see also (\ref{est:uj:K4:0})), the relaxation parameter $\kappa$ in (\ref{est:uj:K1}) can be replaced by the viscous parameter $\mu$.
\item Estimate (\ref{est:uj:K3}) is obtained regardless we consider the hypothesis of initial vacuum or not.
In the context of considering an initial density $\rho_0$ that is always positive, this estimate was already proved in \cite[Theorem~3]{AOKh:2021:Nonl}, by using a different approach.
In view of (\ref{est:uj:K4:0}) below, we can also replace the relaxation parameter $\kappa$ in (\ref{est:uj:K3}) by the viscous parameter $\mu$.
\end{enumerate}
\end{remark}

\begin{proof}(Proposition~\ref{prop:est:j:L2})
For the sake of simplifying the exposition, we shall split the proof into the several enumerated items.
\begin{enumerate}
[leftmargin=0pt,itemindent=*,label=(\arabic*)]
\item
Arguing as we did in the proof of \cite[Theorem~1]{AOKh:2021:Nonl} (see also~\cite[Lemma~1.2]{LS:1978}), and using the maximum principle together with (\ref{bd:rho0:n:j}), we can show that (\ref{est:rho:j:Q}) holds true.
On the other hand, using the solenoidality of $u^j_n$, the fact that $u^j_n=0$ on $\partial\Omega$, together with (\ref{eq:mass:n:j}) and (\ref{ic:u:rho:j})$_3$, we can prove that for all $t\geq 0$ and $1\leq q\leq \infty$ one has
\begin{equation}\label{Res:rhoj:2}
\begin{split}
\left\|\rho^j_n(t)\right\|_{L^q(\Omega)}^q= & \int_0^t\frac{\partial}{\partial s}\left\|\rho^j_n(s)\right\|_{L^q(\Omega)}^q ds + \left\|\rho^j_{0,n}\right\|_{L^q(\Omega)}^q \\
= &
-\int_0^t\int_{\Omega}\nabla\left(|\rho^j_n(s)|^{q}\right)\cdot u^j_n(s)\,dxds + \left\|\rho^j_{0,n}\right\|_{L^q(\Omega)}^q
= \left\|\rho^j_{0,n}\right\|_{L^q(\Omega)}^q.
\end{split}
\end{equation}

\item
Testing (\ref{eq:weak:u:j:nc:s}) with $\psi=u^j_n(t)$, integrating the resulting identity between $0$ and $t\in(0,T)$, using the continuity equation (\ref{eq:mass:n:j}) together with the solenoidality of $u^j_n$, and still using the initial conditions (\ref{ic:u:rho:j}), we obtain
\begin{equation}\label{eq:weak:u:j:nc:s-1}
\begin{split}
&
\frac{1}{2}\left\|\sqrt{\rho^j_n(t)}u^j_n(t)\right\|^2_{L^2(\Omega)}+\frac{\kappa}{2}\left\|\nabla u^j_n(t)\right\|^{2}_{L^2(\Omega)}+\mu\int_{0}^{t}\left\|\nabla u^j_n(s)\right\|^2_{L^2(\Omega)}ds
= \\
&
\frac{1}{2}\left\|\sqrt{\rho^j_{0,n}}u^j_{0,n}\right\|^2_{L^2(\Omega)}+\frac{\kappa}{2}\left\|\nabla u^j_{0,n}\right\|^{2}_{L^2(\Omega)}+
\int_{0}^{t}\int_{\Omega}\rho^j_nf\cdot u^j_n dxds
\end{split}
\end{equation}
The last term of (\ref{eq:weak:u:j:nc:s-1}) is estimated by using the Hölder and Cauchy inequalities,
\begin{equation}\label{eq:weak:u:j:nc:s-2:f}
  \int_{0}^{t}\int_{\Omega}\rho^j_n(s)f(s)\cdot u^j_n(s) dxds \leq
  \frac{1}{4}\int_0^t\left\|\sqrt{\rho^j_n(s)}u^j_n(s)\right\|_{L^2(\Omega)}^2ds + \int_0^t\left\|\sqrt{\rho^j_n(s)}f(s)\right\|_{L^2(\Omega)}^2ds.
\end{equation}
Plugging (\ref{eq:weak:u:j:nc:s-2:f}) into (\ref{eq:weak:u:j:nc:s-1}) and using (\ref{bd:rho0:n:j})-(\ref{bd:un0:j}) and (\ref{bd:Du0:n:j})$_1$, together with (\ref{est:rho:j:Q}), one has
\begin{equation*}
\begin{split}
&
\frac{1}{4}\left\|\sqrt{\rho^j_n(t)}u^j_n(t)\right\|^2_{L^2(\Omega)}+\frac{\kappa}{2}\left\|\nabla u^j_n(t)\right\|^{2}_{L^2(\Omega)}+
\mu\int_{0}^{t}\left\|\nabla u^j_n(s)\right\|^2_{L^2(\Omega)}ds
\leq \\
&
\frac{M^\ast}{2}\left\|u_0\right\|^2_{2,\Omega}+\frac{\kappa}{2}\left\|\nabla u_0\right\|^{2}_{2,\Omega}
+
\frac{1}{2}\int_0^t\left\|\sqrt{\rho^j_n(s)}u^j_n(s)\right\|_{L^2(\Omega)}^2ds +
M^\ast\int_0^t\left\|f(s)\right\|_{L^2(\Omega)}^2ds.
\end{split}
\end{equation*}
Using the Grönwall inequality and taking the supremum in $(0,T)$ in the resulting inequality, we obtain (\ref{est:uj:K1})
for some positive constant $K_1=C\left(M^\ast,\kappa,\left\|u_0\right\|_{2,\Omega},\left\|\nabla u_0\right\|_{2,\Omega},\left\|f\right\|_{L^2(0,T;L^2(\Omega))},T\right)$.

\item
Testing (\ref{eq:weak:u:j:nc:s}) with $\psi=\frac{\partial u^j_n(t)}{\partial t}$, we obtain
\begin{equation}\label{eq:weak:u:j:nc:s-10-1}
\begin{split}
&
\frac{\mu}{2}\frac{d}{dt}\left\|\nabla u^j_n(t)\right\|^{2}_{L^2(\Omega)}+
\left\|\sqrt{\rho^j_n(t)}\frac{\partial u^j_n(t)}{\partial t}\right\|^2_{L^2(\Omega)}+
\kappa\left\|\frac{\partial\nabla u^j_n(t)}{\partial t}\right\|^2_{L^2(\Omega)} = \\
&
J(t)+\int_{\Omega}\rho^j_n(t)f(t)\cdot \frac{\partial u^j_n(t)}{\partial t}\,dx,
\end{split}
\end{equation}
where
\begin{equation*}
J(t):=-\int_{\Omega}\rho^j_n(t)\left(u^j_n(t)\cdot\nabla \right)u^j_n(t)\cdot \frac{\partial u^j_n(t)}{\partial t}dx.
\end{equation*}
To estimate the last term of (\ref{eq:weak:u:j:nc:s-10-1}), we proceed as in (\ref{eq:weak:u:j:nc:s-2:f}) so that
\begin{equation}\label{eq:weak:u:j:nc:s-11}
  \int_{\Omega}\rho^j_n(t)f(t)\cdot \frac{\partial u^j_n(t)}{\partial t}dx \leq
  \frac{1}{2}\left\|\sqrt{\rho^j_n(t)}\frac{\partial u^j_n(t)}{\partial t}\right\|_{L^2(\Omega)}^2 + \frac{1}{2}\left\|\sqrt{\rho^j_n(t)}f(t)\right\|_{L^2(\Omega)}^2.
\end{equation}
For the estimate of $J(t)$, we use (\ref{est:rho:j:Q}) together with the Hölder, Cauchy and Sobolev inequalities so that
\begin{equation}\label{est:J(t):1}
\begin{split}
|J(t)|\leq &
M^\ast\left\|u^j_n(t)\right\|_{L^{d}(\Omega)}\left\|\nabla u^j_n(t)\right\|_{L^2(\Omega)}\left\|\frac{\partial u^j_n(t)}{\partial t}\right\|_{L^{2^\ast}(\Omega)},\qquad 2\leq d\leq 4\\
\leq &
C_1\left\|\nabla u^j_n(t)\right\|^2_{L^2(\Omega)}\left\|\frac{\partial\nabla u^j_n(t)}{\partial t}\right\|_{L^2(\Omega)} \\
\leq &
\frac{\kappa}{2}\left\|\frac{\partial\nabla u^j_n(t)}{\partial t}\right\|_{L^2(\Omega)}^2 +
C_2\left(\sup_{t\in(0,T)}\left\|\nabla u^j_n(t)\right\|^2_{L^2(\Omega)}\right)^2
\end{split}
\end{equation}
for some positive constants $C_1=C(d,M^\ast,\Omega)$ and $C_2=C(M^\ast,\kappa,d,\Omega)$.
%It should be noted that the dependence of these constants on $\Omega$ can be dropped by requiring that $\rho_0\in L^{\frac{2d}{4-d}}(\Omega)$ in the case %of $2\leq d\leq 3$. For $d=4$, it is enough assumption (\ref{cond:r0}), which implies (\ref{est:rho:j:Q}).
Plugging (\ref{eq:weak:u:j:nc:s-11}) and (\ref{est:J(t):1}) into (\ref{eq:weak:u:j:nc:s-10-1}), integrating the resulting inequality between $0$ and $t\in(0,T)$ and using (\ref{ic:u:rho:j}) and the estimate (\ref{est:uj:K1}), we achieve to
\begin{equation*}
\begin{split}
& \mu\left\|\nabla u^j_n(t)\right\|^{2}_{L^2(\Omega)}+\int_0^t\left(\left\|\sqrt{\rho^j_n(s)}\frac{\partial u^j_n}{\partial s}(s)\right\|^2_{L^2(\Omega)}
+\kappa\left\|\frac{\partial\nabla u^j_n(s)}{\partial s}\right\|^2_{L^2(\Omega)}\right)ds \leq \\
&
\mu\left\|\nabla u^j_{0,n}\right\|^{2}_{L^2(\Omega)} + \left\|f(t)\right\|_{L^2(\Omega)}^2 + C,
\end{split}
\end{equation*}
for some positive constant $C=C(M^\ast,\kappa,d,K_1,\Omega,T)$.
Taking the supremum in $(0,T)$ and using (\ref{bd:Du0:n:j})$_1$, there holds
\begin{equation}\label{est:uj:K2:0}
\mu\sup_{t\in(0,T)}\left\|\nabla u^j_n(t)\right\|^{2}_{L^2(\Omega)}+
\int_0^T\left(\left\|\sqrt{\rho^j_n(t)}\frac{\partial u^j_n(t)}{\partial t}\right\|^2_{L^2(\Omega)}+\kappa\left\|\frac{\partial\nabla u^j_n(t)}{\partial t}\right\|^2_{L^2(\Omega)}\right)dt \leq  K_2
\end{equation}
for some positive constant $K_2=C\left(M^\ast,\mu,\kappa,d,\left\|f\right\|_{L^2(0,T;L^2(\Omega))},\Omega,T,K_1\right)$.
In view of (\ref{est:uj:K1}), the relevant information to extract from the estimate (\ref{est:uj:K2:0}) is given by (\ref{est:uj:K2}).

\item
We start by testing (\ref{eq:weak:u:j:nc:s}) with $\psi=\mathbb{A}(u^j_n(t))$, where $\mathbb{A}$ is the Stokes operator considered in (\ref{Stokes:op}),
\begin{equation}\label{est:uj:A:1}
\begin{split}
&
-\kappa\int_{\Omega} \frac{\partial\Delta u^j_n(t)}{\partial t}\cdot\mathbb{A}(u^j_n(t))\,dx
-\mu\int_{\Omega}\Delta u^j_n(t)\cdot\mathbb{A}(u^j_n(t))\,dx  = \\
&
\int_{\Omega}\rho^j_n(t)f(t)\cdot\mathbb{A}(u^j_n(t))\,dx - \int_{\Omega}\rho^j_n(t)\left[\frac{\partial u^j_n(t)}{\partial t}+\left(u^j_n(t)\cdot\nabla\right)u^j_n(t)\right]\cdot \mathbb{A}(u^j_n(t))\,dx.
\end{split}
\end{equation}
Writing $u^j_n$ in the form (\ref{app:sol:uj}) and using the linearity of the Stokes operator (\ref{Stokes:op}), together with (\ref{prop:Leray}) and (\ref{Stokes_bas}), we can show that
\begin{alignat*}{2}
%\label{est:uj:A:11}
& -\kappa\int_{\Omega} \frac{\partial\Delta u^j_n(t)}{\partial t}\cdot\mathbb{A}(u^j_n(t))\,dx = \frac{\kappa}{2}\frac{d}{dt}\int_{\Omega} |\mathbb{A}(u^j_n(t))|^2\,dx, && \\
%\label{est:uj:A:12}
& -\mu\int_{\Omega} \Delta u^j_n(t)\cdot\mathbb{A}(u^j_n(t))\,dx = \mu\int_{\Omega} |\mathbb{A}(u^j_n(t))|^2\,dx. &&
\end{alignat*}
Replacing %(\ref{est:uj:A:11}) and (\ref{est:uj:A:12})
in (\ref{est:uj:A:1}), we get
\begin{equation*}%\label{est:uj:A:2}
\frac{\kappa}{2}\frac{d}{dt}\left\|\mathbb{A}(u^j_n(t))\right\|^2_{L^2(\Omega)} +\mu \left\|\mathbb{A}(u^j_n(t))\right\|^2_{L^2(\Omega)} =
\int_{\Omega}\rho^j_n(t)\left[f(t)-\frac{\partial u^j_n(t)}{\partial t}-\left(u^j_n(t)\cdot\nabla \right)u^j_n(t)\right]\cdot\mathbb{A}(u^j_n(t))\,dx.
\end{equation*}
Using the Hölder, Cauchy and Minkovski inequalities, together with (\ref{est:rho:j:Q}), one has
\begin{equation*} %\label{est:uj:A:3}
\begin{split}
& \kappa\frac{d}{dt}\left\|\mathbb{A}(u^j_n(t))\right\|^2_{L^2(\Omega)} +\mu \left\|\mathbb{A}(u^j_n(t))\right\|^2_{L^2(\Omega)} \leq \\
%& C\left(\left\|\rho^j_n(t)f(t)\right\|^2_{L^2(\Omega)} + \left\|\rho^j_n(t)u_t^j(t)\right\|^2_{L^2(\Omega)} + \left\|\rho^j_n(t)\left(u^j_n(t)\cdot\nabla %\right)u^j_n(t)\right\|^2_{L^2(\Omega)}\right)  \\
& C\left(\left\|f(t)\right\|^2_{L^2(\Omega)} + \left\|\sqrt{\rho^j_n(t)}\frac{\partial u^j_n(t)}{\partial t}\right\|^2_{L^2(\Omega)} +
\int_\Omega|u^j_n(t)|^2|\nabla u^j_n(t)|^2dx\right)
\end{split}
\end{equation*}
for some positive constant $C=C(\mu,M^\ast)$.
Integrating between $0$ and $t\in(0,T)$ and then taking the supremum in $(0,T)$ in the resulting inequality,
\begin{equation}\label{est:uj:A:4}
\begin{split}
& \kappa\sup_{t\in(0,T)}\left\|\mathbb{A}(u^j_n(t))\right\|^2_{L^2(\Omega)} +\mu\int_0^T\left\|\mathbb{A}(u^j_n(t))\right\|^2_{L^2(\Omega)}dt \leq
\kappa\left\|\mathbb{A}(u_{0,n}^j)\right\|^2_{L^2(\Omega)} +
\\
&
%& C\left(\left\|\rho^j_n(t)f(t)\right\|^2_{L^2(\Omega)} + \left\|\rho^j_n(t)u_t^j(t)\right\|^2_{L^2(\Omega)} + \left\|\rho^j_n(t)\left(u^j_n(t)\cdot\nabla %\right)u^j_n(t)\right\|^2_{L^2(\Omega)}\right)  \\
 C\left(\int_0^T\left\|f(t)\right\|^2_{L^2(\Omega)}dt + \int_0^T\left\|\sqrt{\rho^j_n(t)}\frac{\partial u^j_n(t)}{\partial t}\right\|^2_{L^2(\Omega)}dt +
\int_0^T\int_\Omega|u^j_n|^2|\nabla u^j_n|^2dxdt\right).
\end{split}
\end{equation}
We now observe that, in view of \eqref{Stokes:est} from one hand, and \eqref{Sob:ineq:Lap:D2}, \eqref{Stokes:op} and (\ref{ic:u:rho:j:u0}) on the other, there holds
\begin{alignat}{2}
\label{Stokes:est:1:nj}
& \left\|D^2(u^j_n(t))\right\|_{L^2(\Omega)} \leq C_1\left\|\mathbb{A}(u^j_n(t))\right\|^2_{L^2(\Omega)}, && \\
\label{Stokes:est:2:nj}
& \left\|\mathbb{A}(u_{0,n}^j)\right\|^2_{L^2(\Omega)}\leq C_2\left\|\Delta u_{0,n}^j\right\|^2_{L^2(\Omega)}\leq C_3\left\|D^2 u_{0,n}^j\right\|^2_{L^2(\Omega)}  &&
\end{alignat}
for some positive constants $C_1=C(\mu,\Omega)$, $C_2=C(\mu)$ and $C_3=C(\mu,d)$.
Note that in the first inequality of (\ref{Stokes:est:2:nj}) we have used the fact that the Leray projection $\mathbb{P}$ commutes with the Laplacian for the eigenfunctions $\psi_{i}\in H^2(\Omega)\cap V$ (see~\eqref{Stokes_bas}), which in turn can be proved by using the symmetry of $\mathbb{P}$, similarly to \eqref{Stokes_bas}.
Hence, by the application of (\ref{Stokes:est:1:nj})-(\ref{Stokes:est:2:nj}) in (\ref{est:uj:A:4}), there holds
\begin{equation}\label{est:uj:A:5}
\begin{split}
& \kappa\sup_{t\in(0,T)}\left\|D^2u^j_n(t)\right\|^2_{L^2(\Omega)} +\mu\int_0^T\left\|D^2u^j_n(t)\right\|^2_{L^2(\Omega)}dt \leq
 C\left(\left\|D^2u_{0,n}^j\right\|^2_{L^2(\Omega)} +\right.
\\
&
%& C\left(\left\|\rho^j_n(t)f(t)\right\|^2_{L^2(\Omega)} + \left\|\rho^j_n(t)u_t^j(t)\right\|^2_{L^2(\Omega)} + \left\|\rho^j_n(t)\left(u^j_n(t)\cdot\nabla %\right)u^j_n(t)\right\|^2_{L^2(\Omega)}\right)  \\
\left.\int_0^T\left\|f(t)\right\|^2_{L^2(\Omega)}dt + \int_0^T\left\|\sqrt{\rho^j_n(t)}\frac{\partial u^j_n(t)}{\partial t}\right\|^2_{L^2(\Omega)}dt +
\int_0^T\int_\Omega|u^j_n|^2|\nabla u^j_n|^2dxdt\right)
\end{split}
\end{equation}
for some positive constant $C=C(\mu,\kappa,M^\ast,d,\Omega)$.
On the other hand, by combining the Hölder and Cauchy inequalities with the Sobolev inequalities (\ref{Sob:ineq:gen})-(\ref{Moser:ineq:gen}) and (\ref{Sob:ineq:2:D2}), one has
\begin{equation}\label{est:uj:A:5:int}
\begin{split}
\int_0^T\int_\Omega|u^j_n|^2|\nabla u^j_n|^2dxdt\leq &
\int_0^T\|u^j_n(t)\|_{L^{2^\ast}(\Omega)}\|u^j_n(t)\|_{L^q(\Omega)}\|\nabla u^j_n(t)\|_{L^{2}(\Omega)}\|\nabla u^j_n(t)\|_{L^{2^\ast}(\Omega)}dt \\
&
\left(\mbox{for}\ \
\frac{1}{2^\ast}+\frac{1}{q}+\frac{1}{2}+\frac{1}{2^\ast}= 1 \Leftrightarrow q=\frac{2d}{4-d},\ \ 2\leq d\leq 3\right)\\
& \leq
C_1\sup_{t\in(0,T)}\|\nabla u^j_n(t)\|_{L^2(\Omega)}^2\int_0^t\|\nabla u^j_n(t)\|_{L^2(\Omega)}\|D^2 u^j_n(t)\|_{L^2(\Omega)}dt \\
&
\leq
C_2\int_0^t\|\nabla u^j_n(t)\|_{L^2(\Omega)}\|D^2 u^j_n(t)\|_{L^2(\Omega)}dt\\
&
\leq
C_3\int_0^t\|\nabla u^j_n(t)\|_{L^2(\Omega)}^2dt + \frac{\mu}{2C}\int_0^t\|D^2 u^j_n(t)\|_{L^2(\Omega)}^2dt,
\end{split}
\end{equation}
for some positive constants $C_1=C(d,\Omega)$, $C_2=C(\kappa,d,\Omega,K_1)$ and $C_3=C(\mu,\kappa,M^\ast,d,\Omega,K_1)$, and where $C$ is the positive constant from (\ref{est:uj:A:5}).
Plugging \eqref{est:uj:A:5:int} into \eqref{est:uj:A:5} and using \eqref{bd:Du0:n:j}$_2$, together with the estimates \eqref{est:uj:K1} and \eqref{est:uj:K2}, we prove that (\ref{est:uj:K3}) is verified for some positive constant $K_3=C(\mu,\kappa,M^\ast,d,\Omega,\left\|D^2u_0\right\|_{L^2(\Omega)},\break \left\|f\right\|_{L^2(0,T;L^2(\Omega))},K_1,K_2)$.

\item
In this case, we use the Hölder and Cauchy inequalities, together with (\ref{est:rho:j:Q}), to estimate the r.h.s. terms of (\ref{eq:weak:u:j:nc:s-10-1}) as follows,
\begin{alignat}{2}
&
\label{eq:weak:u:j:nc:s-11-1}
  \int_{\Omega}\rho^j_n(t)f(t)\cdot \frac{\partial u^j_n(t)}{\partial t}dx \leq
  \frac{1}{4}\left\|\sqrt{\rho^j_n(t)}\frac{\partial u^j_n(t)}{\partial t}\right\|_{L^2(\Omega)}^2 + \left\|\sqrt{\rho^j_n(t)}f(t)\right\|_{L^2(\Omega)}^2, && \\
& \label{est:J(t):1-1}
|J(t)|\leq
\frac{1}{4}\left\|\sqrt{\rho^j_n(t)}\frac{\partial u^j_n(t)}{\partial t}\right\|^2_{L^2(\Omega)}+M^\ast\int_{\Omega}|u^j_n(t)|^2|\nabla u^j_n(t)|^2dx. &&
\end{alignat}
Plugging (\ref{eq:weak:u:j:nc:s-11-1}) and (\ref{est:J(t):1-1}) into (\ref{eq:weak:u:j:nc:s-10-1}), and integrating the resulting inequality between $0$ and $t\in(0,T)$, we obtain
\begin{equation}\label{eq:weak:u:j:nc:s-12}
\begin{split}
&
\mu\left\|\nabla u^j_n(t)\right\|^{2}_{L^2(\Omega)}+\int_0^t\left(\left\|\sqrt{\rho^j_n(s)}\frac{\partial u^j_n}{\partial s}(s)\right\|^2_{L^2(\Omega)}
+\kappa\left\|\frac{\partial\nabla u^j_n(s)}{\partial s}\right\|^2_{L^2(\Omega)}\right)ds \leq \\
&
\mu\left\|\nabla u^j_{0,n}\right\|^{2}_{L^2(\Omega)} + 2\int_0^t\left\|f(s)\right\|_{L^2(\Omega)}^2ds +
2M^\ast\int_0^t\int_\Omega|u^j_n|^2|\nabla u^j_n|^2dxds.
\end{split}
\end{equation}

\noindent On the other hand, observing that $\partial\Omega\in C^2$, we can use Lemma~\ref{Lem:reg:St} to prove the existence of a unique weak solution $(w,p):$ $w\in H^2(\Omega)$ and $p\in H^1(\Omega)$, with $\int_{\Omega}p(x)dx=0$, for the stationary Stokes problem
\begin{align}
&
\operatorname{div}w =0 \qquad\mbox{in}\quad \Omega, \label{St:inc:n} \\
& -\mu\Delta w + \nabla p = \rho^j_n(t) f(t) -
\left[\rho^j_n(t)\frac{\partial u^j_n(t)}{\partial t} - \rho^j_n(t)\left(u^j_n(t)\cdot\nabla\right)u^j_n(t)\right]
\qquad\mbox{in}\quad \Omega, \label{St:mom:n} \\
& w=0 \qquad\mbox{on}\quad  \partial\Omega. \label{St:bc:n}
\end{align}
Moreover, there exists a positive constant $C_2=C(\mu,\Omega)$ such that
\begin{equation}\label{Stokes:est:r:w}
\left\|w\right\|_{H^2(\Omega)} + \left\| p\right\|_{H^1(\Omega)} \leq
C_2\left\|\rho^j_n(t)f(t) -\left[\rho^j_n(t)\frac{\partial u^j_n(t)}{\partial t} - \rho^j_n(t)\left(u^j_n(t)\cdot\nabla\right)u^j_n(t)\right]\right\|_{L^2(\Omega)}.
\end{equation}
Note that for any $t\in(0,T)$, $w=u^j_n(t)+\sigma \frac{\partial u^j_n(t)}{\partial t}$ and $p=p^j_n(t)$, with $\sigma=\frac{\kappa}{\mu}$, satisfy the Stokes problem (\ref{St:inc:n})-(\ref{St:bc:n}) and therefore (\ref{Stokes:est:r:w}) implies
\begin{equation}\label{Equ:D2}
\begin{split}
&
\mu^2\left\|D^2u^j_n(t)\right\|_{L^2(\Omega)}^2 + \mu\kappa\frac{d}{dt}\left\|D^2u^j_n(t)\right\|_{L^2(\Omega)}^2 + \kappa^2\left\|\frac{\partial D^2u^j_n(t)}{\partial t}\right\|_{L^2(\Omega)}^2\leq \\
&
C_3\left(M^\ast\left\|f(t)\right\|_{L^2(\Omega)}^2 + \left\|\sqrt{\rho^j_n(t)}\frac{\partial u^j_n(t)}{\partial t}\right\|_{L^2(\Omega)}^2 +
M^\ast\int_{\Omega}|u^j_n(t)|^2|\nabla u^j_n(t)|^2dx\right),
\end{split}
\end{equation}
where $C_3=C(\mu,M^\ast,\Omega)$ is a positive constant.
Integrating between $0$ and $t\in(0,T)$, and using (\ref{ic:u:rho:j}), one has
\begin{equation}\label{est:uj+uj_t:St}
\begin{split}
&
\mu^2\int_0^t\left\|D^2u^j_n(s)\right\|_{L^2(\Omega)}^2ds + \mu\kappa\left\|D^2u^j_n(t)\right\|_{L^2(\Omega)}^2 +
\kappa^2\int_0^t\left\|\frac{\partial D^2u^j_n(s)}{\partial s}\right\|_{L^2(\Omega)}^2ds \leq \\
&
\mu\kappa\left\|D^2u^j_{0,n}\right\|_{L^2(\Omega)}^2 +
C_3\left(M^\ast\int_0^t\left\|f(s)\right\|_{L^2(\Omega)}^2ds + \int_0^t\left\|\sqrt{\rho^j_n(s)}\frac{\partial u^j_n}{\partial s}(s)\right\|_{L^2(\Omega)}^2ds \right. \\
& \left.+
M^\ast\int_0^t\int_{\Omega}|u^j_n|^2|\nabla u^j_n|^2dxds\right).
\end{split}
\end{equation}

\noindent Choosing $\delta>0$ so small that $C_3\delta<\frac{1}{2}$, we obtain from (\ref{eq:weak:u:j:nc:s-12}) and (\ref{est:uj+uj_t:St})
\begin{equation}\label{eq:weak:u:j:nc:s-13}
\begin{split}
&
\mu\left\|\nabla u^j_n(t)\right\|^{2}_{L^2(\Omega)}+
\int_0^t\left(\frac{1}{2}\left\|\sqrt{\rho^j_n(s)}\frac{\partial u^j_n}{\partial s}(s)\right\|^2_{L^2(\Omega)}+
2\kappa\left\|\frac{\partial\nabla u^j_n(s)}{\partial s}\right\|^2_{L^2(\Omega)}\right)ds  + \\
&
\delta\left(\mu^2\int_0^t\left\|D^2u^j_n(s)\right\|_{L^2(\Omega)}^2ds + \mu\kappa\left\|D^2u^j_n(t)\right\|_{L^2(\Omega)}^2 + \kappa^2\int_0^t\left\|\frac{\partial D^2u^j_n(s)}{\partial s}\right\|_{L^2(\Omega)}^2ds\right)
\leq \\
&
\mu\left\|\nabla u^j_{0,n}\right\|^{2}_{L^2(\Omega)} +
\delta\mu\kappa\left\|D^2u^j_{0,n}\right\|_{L^2(\Omega)}^2 +
C_4\int_0^t\left\|f(s)\right\|_{L^2(\Omega)}^2ds
+ C_5\int_0^t\int_{\Omega}|u^j_n(s)|^2|\nabla u^j_n(s)|^2dxds
\end{split}
\end{equation}
where $C_4=C(M^\ast)$ and $C_5=C(M^\ast)$ are distinct positive constants.
To estimate the last term, we proceed as in (\ref{est:uj:A:5:int}), but using in the final part the Cauchy inequality with $\delta$.
Hence, we get
\begin{equation}\label{est:uj:A:5:int:0}
\begin{split}
\int_0^t\int_\Omega|u^j_n(s)|^2|\nabla u^j_n(s)|^2dxds\leq &
C_6\int_0^t\|\nabla u^j_n(s)\|_{L^2(\Omega)}\|D^2 u^j_n(s)\|_{L^2(\Omega)}ds\\
\leq
&
C_6\left(\frac{2C_7\mu^2}{\delta}\int_0^t\left\|\nabla u^j_n(s)\right\|_{L^2(\Omega)}^2ds +
\frac{\delta\mu^2}{2C_7}\int_0^t\left\|D^2 u^j_n(s)\right\|_{L^2(\Omega)}^2ds\right),
\end{split}
\end{equation}
where $C_6=\frac{C_2}{\mu}$ and $C_7=C_6C_5$, being $C_2$ given in (\ref{est:uj:A:5:int}) and $C_5$ in (\ref{eq:weak:u:j:nc:s-13}).
Plugging (\ref{est:uj:A:5:int:0}) into (\ref{eq:weak:u:j:nc:s-13}), choosing $\delta$ in the above conditions, taking the supremum in $(0,T)$ of the resulting inequality and using (\ref{bd:Du0:n:j}) and, again, (\ref{est:uj:K1}), we obtain
\begin{equation}\label{est:uj:K4:0}
\begin{split}
&
\mu\sup_{t\in(0,T)}\left(\left\|\nabla u^j_n(t)\right\|^{2}_{L^2(\Omega)}+\kappa\left\|D^2u^j_n(t)\right\|^{2}_{L^2(\Omega)}\right) \\
&
+\int_0^T\left(\left\|\sqrt{\rho^j_n(t)}\frac{\partial u^j_n(t)}{\partial t}\right\|^2_{L^2(\Omega)}+\kappa\left\|\frac{\partial\nabla u^j_n(t)}{\partial t}\right\|^2_{L^2(\Omega)}\right)dt  \\
&
+ \mu^2\int_0^T\left\|D^2u^j_n(t)\right\|_{L^2(\Omega)}^2dt + \kappa^2\int_0^T\left\|\frac{\partial D^2u^j_n(t)}{\partial t}\right\|_{L^2(\Omega)}^2dt
\leq K_4,
\end{split}
\end{equation}
for some positive constant $K_4=C\left(M^\ast,\mu,\kappa,d,\Omega,\left\|u_0\right\|_{2,\Omega},\left\|\nabla u_0\right\|_{2,\Omega},\left\|D^2u _{0}\right\|_{L^2(\Omega)},\left\|f\right\|_{L^2(0,T;L^2(\Omega))},K_1\right)$.
\end{enumerate}
In view of (\ref{est:uj:K1}), (\ref{est:uj:K2}) and (\ref{est:uj:K3}), the relevant information to extract from the estimate (\ref{est:uj:K4:0}) is written in (\ref{est:uj:K4}).
\end{proof}

\begin{remark}\label{rem:K4'}
\rm
Estimate \eqref{est:uj:K4} and part of (\ref{est:uj:K3}) can be obtained by using a slightly different approach.
In fact, testing (\ref{eq:weak:u:j:nc:s}) with $\psi=\mathbb{A}\left(\frac{\partial u_n^j(t)}{\partial t}\right)$, where $\mathbb{A}$ is the Stokes operator considered in (\ref{Stokes:op}), we obtain
\begin{equation*}%\label{0:est:uj:A:2}
\begin{split}
&  \frac{\mu}{2}\frac{d}{dt}\left\|\mathbb{A}(u^j_n(t))\right\|^2_{L^2(\Omega)} +\kappa \left\|\mathbb{A}\left(\frac{\partial u_n^j(t)}{\partial t}\right)\right\|^2_{L^2(\Omega)} = \\
&
\int_{\Omega}\rho^j_n(t)\left[f(t)-\frac{\partial u^j_n(t)}{\partial t}-\left(u^j_n(t)\cdot\nabla \right)u^j_n(t)\right]\cdot\mathbb{A}\left(\frac{\partial u_n^j(t)}{\partial t}\right)\,dx.
\end{split}
\end{equation*}
Arguing as we did for (\ref{est:uj:A:4}), we have
\begin{equation*}%\label{est:uj:A:5:A}
\begin{split}
& \frac{\mu}{2}\sup_{t\in(0,T)}\left\|\mathbb{A}(u^j_n(t))\right\|^2_{L^2(\Omega)}
+\frac{\kappa}{2}\int_0^T\left\|\mathbb{A}\left(\frac{\partial u_n^j(t)}{\partial t}\right)\right\|^2_{L^2(\Omega)}dt \leq
 C\left(\left\|A(u_{0,n}^j)\right\|^2_{L^2(\Omega)} +\right.
\\
&
%& C\left(\left\|\rho^j_n(t)f(t)\right\|^2_{L^2(\Omega)} + \left\|\rho^j_n(t)u_t^j(t)\right\|^2_{L^2(\Omega)} + \left\|\rho^j_n(t)\left(u^j_n(t)\cdot\nabla %\right)u^j_n(t)\right\|^2_{L^2(\Omega)}\right)  \\
\left.\int_0^T\left\|f(t)\right\|^2_{L^2(\Omega)}dt + \int_0^T\left\|\sqrt{\rho^j_n(t)}\frac{\partial u^j_n(t)}{\partial t}\right\|^2_{L^2(\Omega)}dt +
\int_0^T\int_\Omega|u^j_n|^2|\nabla u^j_n|^2dxdt\right),
\end{split}
\end{equation*}
for some positive constant $C=C(\kappa)$.
Next, by using (\ref{Stokes:est:1:nj})-(\ref{Stokes:est:2:nj}) and
\begin{equation*}
%\label{Stokes:est:1:nj:t}
\left\|D^2\left(\frac{\partial u_n^j(t)}{\partial t}\right)\right\|_{L^2(\Omega)} \leq C\left\|\mathbb{A}\left(\frac{\partial u_n^j(t)}{\partial t}\right)\right\|^2_{L^2(\Omega)},
\end{equation*}
where $C=C(\mu,\Omega)$ is a positive constant, we obtain
\begin{equation}\label{0:est:uj:A:4}
\begin{split}
& \mu\sup_{t\in(0,T)}\left\|D^2\left(u^j_n(t)\right)\right\|^2_{L^2(\Omega)} +
\kappa\int_0^T\left\|\frac{\partial D^2\left(u^j_n(t)\right)}{\partial t}\right\|^2_{L^2(\Omega)}dt \leq
C\left(\left\|D^2(u_{0,n}^j)\right\|^2_{L^2(\Omega)} +\right. \\
&
\left.
%& C\left(\left\|\rho^j_n(t)f(t)\right\|^2_{L^2(\Omega)} + \left\|\rho^j_n(t)u_t^j(t)\right\|^2_{L^2(\Omega)} + \left\|\rho^j_n(t)\left(u^j_n(t)\cdot\nabla %\right)u^j_n(t)\right\|^2_{L^2(\Omega)}\right)  \\
\int_0^T\left\|f(t)\right\|^2_{L^2(\Omega)}dt + \int_0^T\left\|\sqrt{\rho^j_n(t)}\frac{\partial u^j_n(t)}{\partial t}\right\|^2_{L^2(\Omega)}dt + \int_0^T\int_\Omega|u^j_n|^2|\nabla u^j_n|^2dxdt\right).
\end{split}
\end{equation}
To estimate the last term, we proceed in a slightly different way than in (\ref{est:uj:A:5:int}),
\begin{equation}\label{est:uj:A:5:int:1}
\begin{split}
\int_0^T\int_\Omega|u^j_n|^2|\nabla u^j_n|^2dxdt\leq &
C_1\sup_{t\in(0,T)}\|\nabla u^j_n(t)\|_{L^2(\Omega)}^2\int_0^T\|\nabla u^j_n(t)\|_{L^2(\Omega)}\|D^2 u^j_n(t)\|_{L^2(\Omega)}dt\\
\leq &
C_2\left(\int_0^T\|\nabla u^j_n(t)\|_{L^2(\Omega)}^2dt+\frac{\mu}{2CC_2T}\int_0^T\|D^2 u^j_n(t)\|_{L^2(\Omega)}^2dt\right) \\
\leq &
C_2\int_0^T\|\nabla u^j_n(t)\|_{L^2(\Omega)}^2dt+\frac{\mu}{2C}\sup_{t\in(0,T)}\|D^2 u^j_n(t)\|_{L^2(\Omega)}^2,
\end{split}
\end{equation}
for some positive constants $C_1=C(d,\Omega)$, $C_2=C(\mu,d,\Omega,T,K_1)$, and where $C$ is the constant from (\ref{0:est:uj:A:4}).
Plugging (\ref{est:uj:A:5:int:1}) into (\ref{0:est:uj:A:4}), and using (\ref{bd:Du0:n:j})$_2$, together with the estimates \eqref{est:uj:K1} and \eqref{est:uj:K2}, we prove that
\begin{equation}\label{est:uj:K4:00}
\mu \sup_{t\in(0,T)}\left\|D^2u^j_n(t)\right\|^2_{L^2(\Omega)} +\kappa\int_0^T\left\|\frac{\partial D^2\left(u^j_n(t)\right)}{\partial t}\right\|^2_{L^2(\Omega)}dt \leq K
\end{equation}
for some positive constant $K=C(\mu,M^\ast,d,\Omega,T,\left\|D^2u_0\right\|_{L^2(\Omega)}, \left\|f\right\|_{L^2(0,T;L^2(\Omega))}, K_1,K_2)$.
Finally, that (\ref{est:uj:K4}) follows from (\ref{est:uj:K4:00}) is immediate.
\end{remark}

Before proceeding with the proof of the estimates (\ref{est:uj:K2'}) and (\ref{est:uj:K4'}), it should be stressed that, despite the reasoning is more direct when the estimate
(\ref{est:uj:K4}) is obtained by the approach described in Remark~\ref{rem:K4'}, we lose some information, when we compare the estimate (\ref{est:uj:K4:00}), which gives rise to it by this method, with the estimate (\ref{est:uj:K4:0}), from which (\ref{est:uj:K4}) follows by the approach used in the proof of Proposition~\ref{prop:est:j:L2}.
More importantly, the method used in the proof of Proposition~\ref{prop:est:j:L2} will be of the utmost importance to prove the estimate (\ref{est:uj:K4'}) below.

\begin{proof}(Proposition~\ref{prop:est:j:Linfty})
Here we also shall split the proof into the two enumerated items.
\begin{enumerate}
[leftmargin=0pt,itemindent=*,label=(\arabic*)]
\item
We rewrite the identity \eqref{eq:weak:u:j:nc:s-10-1} as follows,
\begin{equation}\label{0eq:weak:u:j:nc:s-10:1}
\begin{split}
& \left\|\sqrt{\rho^j_n(t)}\frac{\partial u^j_n(t)}{\partial t}\right\|^2_{L^2(\Omega)}+\kappa\left\|\frac{\partial\nabla u^j_n(t)}{\partial t}\right\|^2_{L^2(\Omega)} = \\
&
-\mu\int_{\Omega}\nabla u^j_n(t):\frac{\partial \nabla u_n^j(t)}{\partial t}\,dx + J(t) + \int_{\Omega}\rho^j_n(t)f(t)\cdot \frac{\partial u^j_n(t)}{\partial t}\,dx.
\end{split}
\end{equation}
Proceeding as we did for \eqref{est:J(t):1}, we get
\begin{equation}\label{0est:J(t):1}
\begin{split}
|J(t)|\leq
\frac{\kappa}{4}\left\|\frac{\partial\nabla u^j_n(t)}{\partial t}\right\|_{L^2(\Omega)}^2 +
C\left(\sup_{t\in(0,T)}\left\|\nabla u^j_n(t)\right\|^2_{L^2(\Omega)}\right)^2,
\end{split}
\end{equation}
for some positive constant $C=C(M^\ast,\kappa,d,\Omega)$.
Using the Cauchy inequality, one has
\begin{equation}\label{0est:J1(t):1}
\left|-\mu\int_{\Omega}\nabla u^j_n(t):\frac{\partial \nabla u_n^j(t)}{\partial t}\,dx\right|\leq
\frac{\kappa}{4}\left\|\frac{\partial\nabla u^j_n(t)}{\partial t}\right\|_{L^2(\Omega)}^2 +
\frac{\mu^2}{\kappa}\left\|\nabla u^j_n(t)\right\|^2_{L^2(\Omega)}.
\end{equation}
Plugging \eqref{eq:weak:u:j:nc:s-11} and \eqref{0est:J(t):1}-\eqref{0est:J1(t):1} into \eqref{0eq:weak:u:j:nc:s-10:1}, we get
\begin{equation}\label{est:uj:K2':1}
\begin{split}
& \frac{1}{2}\left\|\sqrt{\rho^j_n(t)}\frac{\partial u^j_n(t)}{\partial t}\right\|_{L^2(\Omega)}^2 +
\frac{\kappa}{2}\left\|\frac{\partial\nabla u^j_n(t)}{\partial t}\right\|^2_{L^2(\Omega)} \leq \\
&
\frac{\mu^2}{\kappa}\left\|\nabla u^j_n(t)\right\|^2_{L^2(\Omega)} +
C\left(\sup_{t\in(0,T)}\left\|\nabla u^j_n(t)\right\|^2_{L^2(\Omega)}\right)^2 +
\frac{1}{2}\left\|\sqrt{\rho^j_n(t)}f(t)\right\|_{L^2(\Omega)}^2.
\end{split}
\end{equation}
Taking the supremum  in $(0,T)$ of (\ref{est:uj:K2':1}) and using the hypothesis (\ref{Hyp:f:Linfty}), together with the estimates (\ref{est:rho:j:Q}) and (\ref{est:uj:K1}), we show that (\ref{est:uj:K2'}) holds true for some positive constant
$K_2'=C\left(M^\ast,\mu,\kappa,d,\Omega,\left\|f\right\|_{L^\infty(0,T;L^2(\Omega))},K_1\right)$.

\item
To prove this case, we first observe that \eqref{Equ:D2} can be written as follows,
\begin{equation}\label{0:Equ:D2}
\begin{split}
&
\mu^2\left\|D^2u^j_n(t)\right\|_{L^2(\Omega)}^2  + \kappa^2\left\|\frac{\partial D^2u^j_n(t)}{\partial t}\right\|_{L^2(\Omega)}^2\leq
-2\mu\kappa\int_\Omega D^2u^j_n(t):\frac{\partial D^2u^j_n(t)}{\partial t}\,dx  \\
&
+
C_3\left(M^\ast\left\|f(t)\right\|_{L^2(\Omega)}^2 + \left\|\sqrt{\rho^j_n(t)}\frac{\partial u^j_n(t)}{\partial t}\right\|_{L^2(\Omega)}^2 +
M^\ast\int_{\Omega}|u^j_n(t)|^2|\nabla u^j_n(t)|^2dx\right).
\end{split}
\end{equation}
Proceeding as we did for (\ref{est:uj:A:5:int}), we get
\begin{equation}\label{0:Equ:D2:2}
\begin{split}
&
\kappa^2\left\|\frac{\partial D^2u^j_n(t)}{\partial t}\right\|_{L^2(\Omega)}^2\leq \\
&
C_4\left(
\left\|f(t)\right\|_{L^2(\Omega)}^2 +
\left\|\sqrt{\rho^j_n(t)}\frac{\partial u^j_n(t)}{\partial t}\right\|_{L^2(\Omega)}^2 +
\left\|\nabla u^j_n(s)\right\|_{L^2(\Omega)}^2 + \left\|D^2u^j_n(t)\right\|_{L^2(\Omega)}^2\right)
\end{split}
\end{equation}
for some positive constant $C_4=C(\mu,\kappa,M^\ast,d,\Omega,K_1)$.
\end{enumerate}
Taking the supremum of \eqref{0:Equ:D2:2} in $(0,T)$, and using the assumption (\ref{Hyp:f:Linfty}), together with the estimates (\ref{est:rho:j:Q}), \eqref{est:uj:K1}, \eqref{est:uj:K3} and (\ref{est:uj:K2'}), we prove that (\ref{est:uj:K4'}) holds for some positive constant $K_4'=C(\mu,\kappa,M^\ast,d,\Omega,K_1,K_3,K_2'))$.
\end{proof}

\begin{proposition}
Assume $2\leq d\leq 3$ and let $u^j_n$, $\rho^j_n$ and $p^j_n$ be the approximate weak solutions of the problem (\ref{NSV:inc:n})-(\ref{NSV:bc:n}) that have been found in
(\ref{app:sol:uj}), (\ref{app:sol:rhoj}) and (\ref{exist:p:n:j}).
\begin{enumerate}
[leftmargin=0pt,itemindent=*,label=(\arabic*)]
\item If (\ref{mu:k>0}), (\ref{cond:r0}), (\ref{Hyp:u0:V}) and (\ref{Hyp:f:L2}) hold,
then there exists an independent of $j$ (and $n$) positive constant $K_5$ such that
\begin{equation}\label{est:uj:K5:p}
\int_0^T\left\|\nabla p^j_n(t)\right\|_{L^2(\Omega)}^2dt \leq K_5
\end{equation}
\item If (\ref{mu:k>0}), (\ref{cond:r0}), (\ref{Hyp:u0:V}) and (\ref{Hyp:f:Linfty}) hold, then there exists an independent of $j$ (and $n$) positive constant $K_6$ such that
\begin{equation}\label{est:uj:K6:p}
\sup_{t\in[0,T]}\left\|\nabla p^j_n(t)\right\|_{L^2(\Omega)}^2 \leq  K_6.
\end{equation}
\end{enumerate}
\end{proposition}
\begin{proof}
\begin{enumerate}
[leftmargin=0pt,itemindent=*,label=(\arabic*)]
\item
Going back a little bit, we recall the Stokes problem (\ref{St:inc:n})-(\ref{St:bc:n}) and the associated regularity result (\ref{Stokes:est:r:w}).
Similarly to (\ref{Equ:D2}), (\ref{Stokes:est:r:w}) also implies
\begin{equation}\label{Equ:D2:R}
\begin{split}
&
\mu^2\left\|D^2u^j_n(t)\right\|_{L^2(\Omega)}^2 + \mu\kappa\frac{d}{dt}\left\|D^2u^j_n(t)\right\|_{L^2(\Omega)}^2 +
\kappa^2\left\|\frac{\partial D^2u^j_n(t)}{\partial t}\right\|_{L^2(\Omega)}^2
+ \left\|\nabla p^j_n(t)\right\|_{L^2(\Omega)}^2
\leq \\
&
C_3\left(\left\|f(t)\right\|_{L^2(\Omega)}^2 + \left\|\sqrt{\rho^j_n(t)}\frac{\partial u^j_n(t)}{\partial t}\right\|_{L^2(\Omega)}^2 +
\int_{\Omega}|u^j_n(t)|^2|\nabla u^j_n(t)|^2dx\right)
\end{split}
\end{equation}
for some positive constant $C_3=C(\mu,M^\ast,\Omega)$.
Departing from (\ref{Equ:D2:R}), and proceeding exactly in the same way as we did for (\ref{est:uj:K4:0}), we can also show that
\begin{equation}\label{est:uj:K4:0:p}
\begin{split}
&
\mu\sup_{t\in(0,T)}\left(\left\|\nabla u^j_n(t)\right\|^{2}_{L^2(\Omega)}+\kappa\left\|D^2u^j_n(t)\right\|^{2}_{L^2(\Omega)}\right) \\
&
+\int_0^T\left(\left\|\sqrt{\rho^j_n(t)}\frac{\partial u^j_n(t)}{\partial t}\right\|^2_{L^2(\Omega)}+\kappa\left\|\frac{\partial\nabla u^j_n(t)}{\partial t}\right\|^2_{L^2(\Omega)}\right)dt  \\
&
+ \mu^2\int_0^T\left\|D^2u^j_n(t)\right\|_{L^2(\Omega)}^2dt + \kappa^2\int_0^T\left\|\frac{\partial D^2u^j_n(t)}{\partial t}\right\|_{L^2(\Omega)}^2dt
+ \int_0^T\left\|\nabla p^j_n(t)\right\|_{L^2(\Omega)}^2dt
\leq K_4,
\end{split}
\end{equation}
for the same positive constant $K_4$.
Hence, (\ref{est:uj:K5:p}) follows immediately from (\ref{est:uj:K4:0:p}), with $K_5=K_4$ and for $K4$ given in (\ref{est:uj:K4}).

\item
In this case, and similarly to (\ref{0:Equ:D2}), we can write (\ref{Equ:D2:R}) as follows,
\begin{equation}\label{est:uj+uj_t:R-5}
\begin{split}
&
\mu^2\left\|D^2u^j_n(t)\right\|_{L^2(\Omega)}^2  + \kappa^2\left\|\frac{\partial D^2u^j_n(t)}{\partial t}\right\|_{L^2(\Omega)}^2
+
\left\|\nabla p^j_n(t)\right\|_{L^2(\Omega)}^2\leq
-2\mu\kappa\int_\Omega D^2u^j_n(t):\frac{\partial D^2u^j_n(t)}{\partial t}\,dx  \\
&
+
C_3\left(\left\|f(t)\right\|_{L^2(\Omega)}^2 + \left\|\sqrt{\rho^j_n(t)}\frac{\partial u^j_n(t)}{\partial t}\right\|_{L^2(\Omega)}^2 +
\int_{\Omega}|u^j_n(t)|^2|\nabla u^j_n(t)|^2dx\right).
\end{split}
\end{equation}
Using (\ref{est:uj+uj_t:R-5}), and proceeding as we did for (\ref{0:Equ:D2:2}), one has
\begin{equation}\label{est:uj+uj_t:R-6}
\left\|\nabla p^j_n(t)\right\|_{L^2(\Omega)}^2\leq
C_4\left(
\left\|f(t)\right\|_{L^2(\Omega)}^2 +
\left\|\sqrt{\rho^j_n(t)}\frac{\partial u^j_n(t)}{\partial t}\right\|_{L^2(\Omega)}^2 +
\left\|\nabla u^j_n(t)\right\|_{L^2(\Omega)}^2 + \left\|D^2u^j_n(t)\right\|_{L^2(\Omega)}^2\right)
\end{equation}
for some positive constant $C_4=C(\mu,\kappa,M^\ast,d,\Omega,K_1)$.
\end{enumerate}
Finally, justifying as we did to show that (\ref{est:uj:K4'}) is a consequence of \eqref{0:Equ:D2:2}, we can also prove that (\ref{est:uj+uj_t:R-6}) implies (\ref{est:uj:K6:p}), with $K_6=K_4'$ and for $K_4'$ given in (\ref{est:uj:K4'}).
\end{proof}

\section{Passing to the limit $j\to\infty$}\label{Sect:lim:j}

In this section, we perform the proof of Theorem~\ref{thm:e:strong:1}, which in fact has been started in the previous sections.

\begin{proof}(Theorem~\ref{thm:e:strong:1})
Due to (\ref{est:uj:K1})-(\ref{est:uj:K4'}), we can use the Banach-Alaoglu theorem to extract subsequences (still labelled by the superscript $j$) such that
\begin{alignat}{2}
\label{wc:conv:H1:2:j}
& u^j_n\xrightharpoonup[j\to\infty]{} u_n\quad \text{in}\ L^{2}(0,T;V),\qquad
&& u^j_n\xrightharpoonup[j\to\infty]{\ast} u_n\quad \text{in}\ L^{\infty}(0,T;V), \\
\label{wc:conv:H2:2:j}
& u^j_n\xrightharpoonup[j\to\infty]{} u_n\quad \text{in}\ L^{2}(0,T;H^2(\Omega)),\qquad
&& u^j_n\xrightharpoonup[j\to\infty]{\ast} u_n\quad \text{in}\ L^{\infty}(0,T;H^2(\Omega)),  \\
\label{wc:conv:H1:2:t:j}
& \frac{\partial u^j_n}{\partial t}\xrightharpoonup[j\to\infty]{} \frac{\partial u_n}{\partial t}\quad \text{in}\ L^{2}(0,T;V),\qquad
&& \frac{\partial u^j_n}{\partial t}\xrightharpoonup[j\to\infty]{\ast} \frac{\partial u_n}{\partial t}\quad \text{in}\ L^{\infty}(0,T;V),  \\
\label{wc:conv:H2:2:t:j}
& \frac{\partial u^j_n}{\partial t}\xrightharpoonup[j\to\infty]{} \frac{\partial u_n}{\partial t}\quad \text{in}\ L^{2}(0,T;H^2(\Omega)),\qquad
&& \frac{\partial u^j_n}{\partial t}\xrightharpoonup[j\to\infty]{\ast} \frac{\partial u_n}{\partial t}\quad \text{in}\ L^{\infty}(0,T;H^2(\Omega)),  \\
\label{wc:conv:H1:press:j}
& p^j_n\xrightharpoonup[j\to\infty]{} p_n\quad \text{in}\ L^{2}(0,T;W^{1,2}(\Omega)),\qquad
&& p^j_n\xrightharpoonup[j\to\infty]{\ast} p_n\quad \text{in}\ L^{\infty}(0,T;W^{1,2}(\Omega)).
\end{alignat}

Then, we can use the Aubin-Lions compactness lemma (see (\ref{Aubin:1}) in Lemma~\ref{lem:AL}), together with (\ref{wc:conv:H1:2:j}), (\ref{wc:conv:H1:2:t:j}) and the compact embedding $W^{1,2}_0(\Omega)\hookrightarrow\hookrightarrow L^2(\Omega)$,  so that
\begin{equation}\label{str:conv:u:j}
u^j_n\xrightarrow[j\to\infty]{} u_n\quad\mbox{in}\ L^2(0,T;L^2(\Omega))
\end{equation}
for some subsequence still labelled by $u^j_n$.
On the other hand, in view of (\ref{est:rho:j:Q}), the Banach-Alaoglu theorem also allows us to extract a subsequence (still labelled by $\rho^j_n$) such that
\begin{equation}\label{wc*:conv:rho:j}
\rho^j_n\xrightharpoonup[j\to\infty]{\star}\rho_n\quad \mbox{in}\ \ L^\infty(0,T;L^\infty(\Omega)).
\end{equation}
Moreover, $\rho_n$ satisfies
\begin{alignat}{2}
\label{NSV:mass:nn}
&
\frac{\partial\rho_n}{\partial t}+\operatorname{div}(\rho_n u_n) =0 \qquad\mbox{in}\quad {Q_T}, && \\
\label{est:rho:n}
&
\frac{1}{n}\leq\rho_n\leq M^\ast<\infty \qquad\mbox{in} \ \ {Q_T}.
\end{alignat}
In addition, using the fact that $\operatorname{div}u_n =0$ in $Q_T$,
together with (\ref{NSV:mass:nn}), and with the estimates (\ref{est:rho:j:Q}) and (\ref{est:uj:K1}), it can be proved that
\begin{equation}\label{ub:rho':j}
\frac{\partial\rho^j_n}{\partial\,t}\ \ \mbox{is uniformly bounded
in}\ L^2(0,T;W^{-1,2^\ast}(\Omega)).
\end{equation}
Moreover the following compact and continuous embeddings  hold,
\begin{equation}\label{conp:cont:imb:r}
L^\infty(\Omega)\hookrightarrow\hookrightarrow W^{-1,\infty}(\Omega) \hookrightarrow W^{-1,2^\ast}(\Omega).
\end{equation}
Then, (\ref{est:rho:j:Q}), (\ref{ub:rho':j}) and (\ref{conp:cont:imb:r}) allow us to use the Aubin-Lions compactness lemma
(see (\ref{Aubin:2}) in Lemma~\ref{lem:AL}) so that, for some subsequence (still labelled by $\rho^j_n$)
\begin{equation}\label{str:conv:1:rho:j}
\rho^j_n\xrightarrow[j\to\infty]{}\rho_n\quad\mbox{in}\ C([0,T];W^{-1,\infty}(\Omega)).
\end{equation}
On the other hand, it follows from (\ref{NSV:ic:n})$_2$ and (\ref{ic:u:rho:j})$_3$, together with (\ref{Res:rhoj:2}),  that
\begin{equation}\label{rn(0)=r(0)}
\left\|\rho^j_n(t)\right\|_{L^2(\Omega)}^{2}=
\left\|\rho_{0,n}\right\|_{L^2(\Omega)}^{2}\quad\mbox{and}\quad \left\|\rho_n(t)\right\|_{L^2(\Omega)}^{2}=
\left\|\rho_{0,n}\right\|_{L^2(\Omega)}^{2}
\quad\forall\ t\in[0,T].
\end{equation}
Thus, applying (\ref{wc*:conv:rho:j}) and (\ref{rn(0)=r(0)}), together with (\ref{str:conv:1:rho:j}), we get for all $t\in[0,T]$
\begin{equation}\label{rn->r:2}
\left\|\rho^j_n(t)-\rho_n(t)\right\|_{L^2(\Omega)}^{2}=\left\|\rho^j_n(t)\right\|_{L^2(\Omega)}^{2}-\left\|\rho_n(t)\right\|_{L^2(\Omega)}^{2}+2\int_{\Omega}(\rho_n(t)-\rho^j_n(t))\rho(t)\,dx
\xrightarrow[j\to\infty]{} 0.
\end{equation}
As a consequence of (\ref{est:rho:n}) and (\ref{rn->r:2}), we have
\begin{equation}\label{rn->r:l}
\left\|\rho^j_n(t)-\rho_n(t)\right\|_{L^2(\Omega)}^{2}\xrightarrow[j\to\infty]{} 0    \quad\forall\ q: 2\leq q<\infty.
\end{equation}
Hence, (\ref{str:conv:1:rho:j}) and (\ref{rn->r:l}) assure that
\begin{equation*}
\rho^j_n\xrightarrow[j\to\infty]{}\rho_n\quad\mbox{in}\ C([0,T];L^{q}(\Omega))   \quad\forall\ q: 2\leq q<\infty,
\end{equation*}
From this and (\ref{est:rho:j:Q}), one has
\begin{equation}\label{str:conv:rho:j}
\rho^j_n\xrightarrow[j\to\infty]{}\rho_n\quad\mbox{in}\ C([0,T];L^{q}(\Omega))   \quad\forall\ q\geq 1.
\end{equation}
By the application of (\ref{wc:conv:H1:2:t:j}) and (\ref{str:conv:u:j}), together with (\ref{est:rho:n}) and (\ref{str:conv:rho:j}), we have
\begin{alignat}{2}
\label{wc:conv:rvt}
& \rho^j_n \frac{\partial u^j_n}{\partial t}\xrightharpoonup[j\to\infty]{}\rho_n \frac{\partial u_n}{\partial t}\quad\mbox{in}\ L^2(0,T;L^2(\Omega)), && \\
\label{str:conv:rv}
& \rho^j_n u^j_n\xrightarrow[j\to\infty]{}\rho_n u_n\quad\mbox{in}\ L^r(0,T;L^r(\Omega)),\ \mbox{with}\ 1\leq r<2^\ast. &&
\end{alignat}
Gathering the information of (\ref{est:rho:j:Q}), (\ref{est:uj:K1}) and (\ref{est:uj:K2}) with (\ref{str:conv:u:j}) and (\ref{str:conv:rho:j}), we can prove that
\begin{equation}\label{wc:rho:conv:j}
\rho^j_n(u^j_n\cdot\nabla )u^j_n\xrightarrow[j\to\infty]{}\rho_n(u_n\cdot\nabla)u_n\quad\text{in}\ \ L^1(0,T;L^1(\Omega)).
\end{equation}
In addition, due to (\ref{wc:conv:H2:2:j}) and (\ref{wc:conv:H2:2:t:j}), we also have
\begin{alignat}{2}
\label{wc:conv:Delta:u:2:j}
& \Delta u^j_n\xrightharpoonup[j\to\infty]{} \Delta u_n\quad \text{in}\ L^{2}(0,T;L^2(\Omega)), && \\
 \label{wc:conv:Delta:ut:2:j}
& \frac{\partial\Delta u^j_n}{\partial t}\xrightharpoonup[j\to\infty]{} \frac{\partial \Delta u_n}{\partial t}\quad \text{in}\ L^{2}(0,T;L^2(\Omega)). &&
\end{alignat}
Let now $\zeta\in C^\infty_0([0,T))$.
Multiplying (\ref{eq:weak:u:j:nc:s}) by $\zeta$ and integrating the resulting equation between $0$ and $T$, we obtain
\begin{equation}\label{eq:lim1:nz}
\int_{Q_T}\left[\rho^j_n \frac{\partial u^j_n}{\partial t} +\rho^j_n \left(u^j_n \cdot\nabla \right)u^j_n -\mu\Delta u^j_n -\kappa\frac{\partial\Delta u^j_n}{\partial t} \right]
\cdot\psi\,\zeta\,dxdt =
\int_{Q_T}\rho^j_n f \cdot\psi\,\zeta\,dxdt
\end{equation}
for all $\psi\in H^2(\Omega)\cap V$.
Proceeding similarly for (\ref{eq:weak:u:j:nc:s:p}), we get
\begin{equation}\label{eq:lim1:nz:p}
\begin{split}
& \int_{Q_T}\left[\rho^j_n \frac{\partial u^j_n}{\partial t} +\rho^j_n \left(u^j_n \cdot\nabla \right)u^j_n -\mu\Delta u^j_n -\kappa\frac{\partial\Delta u^j_n}{\partial t} \right]
\cdot\psi\,\zeta\,dxdt - \int_{Q_T}\rho^j_n f \cdot\psi\,\zeta\,dxdt = \\
&
\int_{Q_T}p^j_n \operatorname{div}\psi\,\zeta\,dxdt
\end{split}
\end{equation}
for all $\psi\in H^2(\Omega)\cap W^{1,2}_0(\Omega)$.
For the same function $\zeta$, we multiply (\ref{eq:mass:n:j}) by $\eta=\phi\,\zeta$, with $\phi\in C^\infty_0({\Omega})$, and integrate the resulting equation over $Q$ so that
\begin{equation}\label{eq:lim2:nz}
-
\int_{Q_T}\rho^j_n\phi\,\zeta'\,dxdt
-
\int_{Q_T}\rho^j_nu^j_n\cdot\nabla \phi\,\zeta\,dxdt =
\zeta(0)\int_{\Omega}\rho_{0,n}^j\phi\,dx
\qquad \forall\ \phi\in C^\infty_0({\Omega}).
\end{equation}
Then, we use, for each corresponding term of (\ref{eq:lim1:nz}), the convergence results (\ref{wc:conv:rvt}), (\ref{wc:rho:conv:j}), (\ref{wc:conv:Delta:u:2:j})-(\ref{wc:conv:Delta:ut:2:j}), together with (\ref{str:conv:rho:j}), to pass the equation (\ref{eq:lim1:nz}) to the limit $j\to\infty$.
To pass (\ref{eq:lim1:nz:p}) to the limit $j\to\infty$, we use, in addition, (\ref{wc:conv:H1:press:j}).

The passage of the equation (\ref{eq:lim2:nz}) to the limit $j\to\infty$ uses (\ref{str:conv:rho:j}) and (\ref{str:conv:rv}) together with (\ref{sc:rn0}).
After all, we obtain
\begin{equation}\label{eq:lim:j:mom:n}
\int_{Q_T}\left[\rho_n \frac{\partial u_n}{\partial t} +\rho_n \left(u_n \cdot\nabla \right)u_n -\mu\Delta u_n
-\kappa\frac{\partial\Delta u_n}{\partial t} \right]
\cdot\psi\zeta\,dx =
\int_{Q_T}\rho_n f \cdot\psi\zeta\,dxdt
\end{equation}
for all $\psi\in H^2(\Omega)\cap V$ and all $\zeta\in C^\infty_0([0,T))$,
\begin{equation}\label{eq:lim:j:mom:p:n}
\int_{Q_T}\left[\rho_n \frac{\partial u_n}{\partial t} +\rho_n\left(u_n \cdot\nabla \right)u_n -\mu\Delta u_n -\kappa \frac{\partial\Delta u_n}{\partial t} \right]
\cdot\psi\,dxdt - \int_{\Omega}\rho_n f \cdot\psi\,\zeta\,dxdt =
\int_{Q_T}p_n \operatorname{div}\psi\,\zeta\,dx
\end{equation}
for all $\psi\in H^2(\Omega)\cap W^{1,2}_0(\Omega)$ and all $\zeta\in C^\infty_0([0,T))$,
and
\begin{equation}\label{eq:lim:j:mass:n}
-\int_{Q_T}\rho_n\,\eta\,\zeta'\,dxdt
-\int_{Q_T}\rho_n u_n\cdot\nabla \eta\,\zeta\,dxdt =
\zeta(0)\int_{Q_T}\rho_{0,n}\eta\,dx
\end{equation}
for all $\phi\in C^\infty_0({\Omega})$ and all $\zeta\in C^\infty_0([0,T))$.

\section{Passing to the limit $n\to\infty$}\label{Sect:lim:n}

In the last section, we have proved that for each $n\in\mathds{N}$ there exists, at least, a solution $(\rho_n,u_n,p_n)$ such that
\begin{equation}\label{eq:lim:j:mom}
\int_{\Omega}\left[\rho_n \frac{\partial u_n}{\partial t} +\rho_n \left(u_n \cdot\nabla \right)u_n -\mu\Delta u_n
-\kappa\frac{\partial\Delta u_n}{\partial t} \right]\cdot\psi\,dx =
\int_{\Omega}\rho_n f \cdot\psi\,dx
\end{equation}
for all $\psi\in H^2(\Omega)\cap V$, and
\begin{equation}\label{eq:lim:j:mom:p}
\int_{\Omega}\left[\rho_n \frac{\partial u_n}{\partial t} +\rho_n \left(u_n \cdot\nabla \right)u_n -\mu\Delta u_n -\kappa\frac{\partial\Delta u_n}{\partial t} \right]\cdot\psi\,dx - \int_{\Omega}\rho_n f \cdot\psi\,dx =
\int_{\Omega}p_n \operatorname{div}\psi\,dx
\end{equation}
for all $\psi\in H^2(\Omega)\cap W^{1,2}_0(\Omega)$, and both in the distribution sense on $(0,T)$.
Moreover, from (\ref{NSV:mass:nn}), we easily realize that $\rho_n$ satisfies
\begin{equation}\label{NSV:mass:lim:j}
\frac{\partial\rho_n}{\partial t}+u_n\cdot\nabla\rho_n=0 \qquad\mbox{in}\quad Q.
\end{equation}

Using the identities (\ref{eq:lim:j:mom}), (\ref{eq:lim:j:mom:p}) and (\ref{NSV:mass:lim:j}), we can proceed as we did for (\ref{est:rho:j:Q}), (\ref{est:uj:K1})-(\ref{est:uj:K4'}) and (\ref{est:uj:K5:p})-(\ref{est:uj:K6:p}), using in this case  (\ref{bd:u0:n}) and (\ref{bd:Du0:n}),  to show that
\begin{alignat}{2}
\label{est:rho:n:Q}
&
0<\frac{1}{n}\leq\inf_{x\in\overline{\Omega}}\rho_{0,n}(x)\leq\rho_n(x,t)\leq\sup_{x\in\overline{\Omega}}\rho_{0,n}(x)\leq M^\ast<\infty
\qquad\forall\ (x,t)\in Q_T, \\
\label{est:un:K1}
&
\sup_{t\in(0,T)}\left(\left\|\sqrt{\rho_n(t)}u_n(t)\right\|^2_{L^2(\Omega)}
+
\kappa\left\|\nabla u_n(t)\right\|^{2}_{L^2(\Omega)}\right)
+
\mu\int_{0}^{T}\left\|\nabla u_n(t)\right\|^2_{L^2(\Omega)}dt \leq K_1, && \\
\label{est:un:K2}
&
\int_0^T\left(\left\|\sqrt{\rho_n(t)}\frac{\partial u_n(t)}{\partial t}\right\|^2_{L^2(\Omega)}
+
\kappa\left\|\frac{\partial\nabla u_n(t)}{\partial t}\right\|^2_{L^2(\Omega)}\right)dt \leq  K_2, && \\
&
\label{est:un:K2'}
\sup_{t\in[0,T]}\left(\left\|\sqrt{\rho_n(t)}\frac{\partial u_n(t)}{\partial t}\right\|^2_{L^2(\Omega)}
+
\kappa\left\|\frac{\partial\nabla u_n(t)}{\partial t}\right\|^2_{L^2(\Omega)}\right)\leq K_2' && \\
\label{est:un:K3}
&
\kappa\sup_{t\in(0,T)}\left\|D^2u_n(t)\right\|^2_{L^2(\Omega)}
+
\mu\int_0^T\left\|D^2u_n(t)\right\|^2_{L^2(\Omega)}dt \leq K_3, && \\
\label{est:un:K4}
&
\kappa^2\int_0^T\left\|\frac{\partial D^2u_n(t)}{\partial t}\right\|_{L^2(\Omega)}^2dt
\leq K_4 && \\
\label{est:un:K4'}
&
\kappa^2\sup_{t\in[0,T]} \left\|\frac{\partial D^2u_n(t)}{\partial t}\right\|_{L^2(\Omega)}^2\leq K_4', && \\
\label{est:un:K7}
&
\int_0^T\left\|\nabla p^j_n(t)\right\|_{L^2(\Omega)}^2dt \leq K_7, && \\
&
\label{est:un:K8}
\sup_{t\in[0,T]}\left\|\nabla p^j_n(t)\right\|_{L^2(\Omega)}^2 \leq  K_8. &&
\end{alignat}

and where the positive constants $K_1$, $K_2$, $K_2'$, $K_3$, $K_4$ and $K_4'$ do not depend on $n$.
Then, due to  (\ref{est:un:K1}), (\ref{est:un:K2}), (\ref{est:un:K2'}), (\ref{est:un:K3}), (\ref{est:un:K4}),  (\ref{est:un:K4'}), (\ref{est:un:K7}) and (\ref{est:un:K8}), we can use the Banach-Alaoglu theorem to extract subsequences (still labelled by the subscript $n$) such that
\begin{alignat}{2}
\label{wc:conv:H1:2:n}
& u_n\xrightharpoonup[n\to\infty]{} u\quad \text{in}\ L^{2}(0,T;V), \qquad
&& u_n\xrightharpoonup[n\to\infty]{\ast} u\quad \text{in}\ L^{\infty}(0,T;V),  \\
\label{wc:conv:H2:2:n}
& u_n\xrightharpoonup[n\to\infty]{} u\quad \text{in}\ L^{2}(0,T;H^2(\Omega)), \qquad
&& u_n\xrightharpoonup[n\to\infty]{\ast} u\quad \text{in}\ L^{\infty}(0,T;H^2(\Omega)), \\
\label{wc:conv:H1:2:t:n}
& \frac{\partial u_n}{\partial t}\xrightharpoonup[n\to\infty]{} \frac{\partial u}{\partial t}\quad \text{in}\ L^{2}(0,T;V), \qquad
&& \frac{\partial u_n}{\partial t}\xrightharpoonup[n\to\infty]{\ast} \frac{\partial u}{\partial t}\quad \text{in}\ L^{\infty}(0,T;V),  \\
\label{wc:conv:H2:2:t:n}
& \frac{\partial u_n}{\partial t}\xrightharpoonup[n\to\infty]{} \frac{\partial u}{\partial t}\quad \text{in}\ L^{2}(0,T;H^2(\Omega)), \qquad
&& \frac{\partial u_n}{\partial t}\xrightharpoonup[n\to\infty]{\ast} \frac{\partial u}{\partial t}\quad \text{in}\ L^{\infty}(0,T;H^2(\Omega)),  \\
\label{wc:conv:W12:2:p:n}
& p_n\xrightharpoonup[n\to\infty]{} p\quad \text{in}\ L^{2}(0,T;W^{1,2}(\Omega)), \qquad
&& p_n\xrightharpoonup[n\to\infty]{\ast} p\quad \text{in}\ L^{\infty}(0,T;W^{1,2}(\Omega)).
\end{alignat}
Observing (\ref{wc:conv:H1:2:n}) and (\ref{wc:conv:H1:2:t:n}), we can use the Aubin-Lions compactness lemma (see (\ref{Aubin:1}) in Lemma~\ref{lem:AL}), so that for some subsequence (still labelled by the subscript $n$)
\begin{equation}\label{sc:nabla:un}
\nabla u_n \xrightarrow[n\to\infty]{} \nabla u\quad \mbox{in} \ \ L^2(0,T;L^2(\Omega)).
\end{equation}
On the other hand, by using (\ref{NSV:mass:nn}), together with (\ref{est:rho:n:Q}) and (\ref{est:un:K1}),
it can be proved that
\begin{equation}\label{ub:rho':n}
\frac{\partial\rho_n}{\partial\,t}\, \ \mbox{is uniformly bounded in}\ L^2(0,T;W^{-1,2}(\Omega)).
\end{equation}
Now, due to (\ref{est:rho:n:Q}) and (\ref{ub:rho':n}), we can use the Aubin-Lions compactness lemma (see (\ref{Aubin:2}) in Lemma~\ref{lem:AL}) so that, for some subsequence (still labelled by $\rho_n$),
\begin{equation}\label{str:conv:1:rho:n}
\rho_n\xrightarrow[n\to\infty]{}\rho\quad\mbox{in}\ C([0,T];W^{-1,2}(\Omega)).
\end{equation}
Hence,
\begin{equation}\label{str:conv:2:rho,u:n}
\rho_nu_n\xrightarrow[n\to\infty]{}\rho u\quad\mbox{in}\ (C_0([0,T]\times\Omega))'
\end{equation}
$\rho$ is a solution of (\ref{NSV:mass}) and
\begin{alignat}{2}
\label{est:rho}
&
0\leq\rho\leq M^\ast<\infty \qquad\mbox{in} \ \ {Q_T}, && \\
%\label{equ:rho}
&
\|\rho(t)\|_{L^q(\Omega)}=\|\rho_0\|_{L^q(\Omega)}\quad \forall\ q: 1\leq q\leq \infty. \nonumber &&
\end{alignat}
However, contrary to the case when $\rho_0$ is bounded away from $0$, in this work $u_n$ is not bounded in $L^2(0,T;L^2(\Omega))$.
This brings us much more difficulty in the passage to the limit $n\to\infty$, which is overcame by the following result.

\begin{lemma}\label{lemm:DiP-L}
Let the conditions of Theorem~\ref{thm:e:strong:1} be fulfilled and assume that (\ref{str:conv:1:rho:n}) and (\ref{str:conv:2:rho,u:n}) hold.
Then there exist subsequences (still labelled by the subscript $n$) such that
\begin{alignat}{2}
\label{sc:rho:n:Lp}
& \rho_n\xrightarrow[n\to\infty]{}\rho\quad\mbox{in}\ C([0,T];L^q(\Omega))\quad \forall\ q\geq 1, && \\
&
\label{sc:u:rho:n:L2}
\sqrt{\rho_n}u_n\xrightarrow[n\to\infty]{}\sqrt{\rho}u\quad\mbox{in}\ L^2(0,T;L^2(\Omega)). &&
\end{alignat}
\end{lemma}
\begin{proof}
The proof of Lemma~\ref{lemm:DiP-L} uses the DiPerna-Lions~\cite{DiP-L:1989} theory for linear transport equations combined with renormalization arguments. For the proof see~\cite[Theorem~2.5]{Lions:1996} and Desjardins~\cite{Desjardins:1997}.
\end{proof}

In order to apply this result, we observe that combining (\ref{NSV:mass:nn}) and (\ref{NSV:inc:n}) with (\ref{eq:lim:j:mom:p:n}) and
(\ref{eq:lim:j:mass:n}), we can show that
\begin{equation}\label{eq:lim:j:mom:p:t}
\begin{split}
& -\int_{Q_T}\left(\rho_n u_n - \kappa\Delta u_n\right)\cdot\frac{\partial \varphi}{\partial t}dxdt -
\mu\int_{Q_T} \Delta u_n\cdot\varphi\,dxdt - \int_{Q_T}\rho_n u_n\otimes u_n : \nabla\varphi\,dxdt \\
&-
\int_{Q_T} \rho_n f\cdot\varphi\,dxdt
=
 \int_{\Omega}\big(\rho_{n,0}u_{n,0}\cdot\varphi(0) + \kappa \nabla u_{n,0} :\nabla\varphi(0)\big)dx +
\int_{Q_T}p_n \operatorname{div}\varphi\,dx
\end{split}
\end{equation}
for all $\varphi\in C^1_0([0,T);W^{1,2}_0(\Omega)\cap H^2(\Omega))$,
and
\begin{equation}\label{eq:lim:j:mass:t}
-\int_{Q_T}\rho_n\frac{\partial\phi}{\partial t}\,dxdt
-\int_{Q_T}\rho_n u_n\cdot\nabla \phi\,dxdt =
\int_{\Omega}\rho_{0,n}\phi(0)\,dx
\end{equation}
for all $\phi\in C^\infty_0([0,T)\times\Omega)$.
We can proceed for (\ref{eq:lim:j:mom:n}) in the same way as we did for (\ref{eq:lim:j:mom:p:t}).

Now, we can use (\ref{wc:conv:H1:2:n})-(\ref{wc:conv:H2:2:n}), (\ref{wc:conv:W12:2:p:n}), (\ref{sc:nabla:un}), (\ref{sc:rho:n:Lp}) and (\ref{sc:u:rho:n:L2}), together with (\ref{moll:u0}), (\ref{moll:rho0}) and (\ref{moll:2}) to pass the integral identity (\ref{eq:lim:j:mom:p:t}) to the limit $n\to\infty$ so that, in view of the regularity of $\rho$, $u$ and $p$,
(\ref{NSV:mom}) holds almost everywhere in $Q_T$.
By using (\ref{sc:rho:n:Lp}) and (\ref{sc:u:rho:n:L2}), together with (\ref{moll:rho0}) and (\ref{moll:2}) , we can also pass the integral identity (\ref{eq:lim:j:mass:t}) to the limit $n\to\infty$ so that, in view of the regularity of $\rho$ and $u$, (\ref{NSV:mass}) holds almost everywhere in $Q_T$.

Finally, combining the estimates (\ref{est:rho:n:Q}) and (\ref{est:un:K1})-(\ref{est:un:K4'}) with the convergence results (\ref{wc:conv:H1:2:n})-(\ref{wc:conv:W12:2:p:n}) and (\ref{sc:rho:n:Lp})-(\ref{sc:u:rho:n:L2}), we can show that the enumerated items (1)-(5) of Theorem~\ref{thm:e:strong:1} and (1)-(4) of Theorem~\ref{thm:e:strong:2} hold true.
\end{proof}

\begin{remark}\label{C-alpha:reg}
In addition to the estimates (\ref{est:un:K3}) and (\ref{est:un:K4}), we can see that, by combining the Sobolev inequalities (\ref{Sob:ineq:Lap:D2}) and (\ref{Sob:ineq:infty:Lap}) with these estimate, we easily obtain
\begin{alignat*}{2}
%\label{est:uj:K3'}
&
\kappa\sup_{t\in(0,T)}\left\|u_n(t)\right\|^2_{C^{0,\alpha}(\overline{\Omega})}\leq K_7,\quad 0<\alpha\leq 2-\frac{d}{2},\ 2\leq d\leq 3, && \\
&
\kappa^2\int_0^T\left\|\frac{\partial u_n(t)}{\partial t}\right\|^2_{C^{0,\alpha}(\overline{\Omega})}dt\leq K_8,\quad 0<\alpha\leq 2-\frac{d}{2},\ 2\leq d\leq 3, &&
\end{alignat*}
for some positive constants $K_7=C(d,K_3)$ $K_8=C(d,K_4)$.
\end{remark}

\section{Proof of Theorem~\ref{thm:u:strong}}\label{Sect:unique}

In this section, we shall prove the uniqueness of the components $u$ and $\rho$ of a solution $(u,\rho,p)$ of the problem (\ref{NSV:inc})-(\ref{NSV:u:bc}).
Before doing so, let us invoke some results that provides us with further regularity on the solutions $(u,\rho,p)$ of the problem (\ref{NSV:inc})-(\ref{NSV:u:bc}).

We firstly recall a result about the regularity of the density $\rho$ for the nonhomogeneous Navier-Stokes equations.
This result is still useful here, because the continuity equation is the same for both nonhomogeneous Navier-Stokes and Navier-Stokes-Voigt systems of
equations.

\begin{proposition}\label{lemma:LS}
Let $(u,\rho,p)$ be a solution of the problem (\ref{NSV:inc})-(\ref{NSV:u:bc}) in the conditions of Theorem~\ref{thm:e:strong:2}.
If (\ref{hyp:gra:ro}) and
\begin{equation}\label{u:C(0,T;W1-infty)}
u\in C([0,T];W^{1,\infty}(\Omega))
\end{equation}
hold, then
\begin{alignat}{2}
& \label{eq:uniqs1}
\|\nabla \rho(t)\|_{L^\infty(\Omega)}\leq
\sqrt{d}\|\nabla \rho_0\|_{L^\infty(\Omega)}
\exp\left(\int_{0}^{t}\|\nabla u(s)\|_{L^\infty(\Omega)}\,ds\right), && \\
& \label{eq:uniqs2}
\|\rho_t(t)\|_{L^\infty(\Omega)}\leq
\sqrt{d}\|\nabla \rho_0\|_{L^\infty(\Omega)}
\|u(t)\|_{L^\infty(\Omega)}
\exp\left(\int_{0}^{t}\|\nabla u(s)\|_{L^\infty(\Omega)}\,ds\right) &&
\end{alignat}
for all $t\in[0,T]$.
\end{proposition}
\begin{proof}
We address the proof of Proposition~\ref{lemma:LS} to Ladyzhenskaya and Solonnikov~\cite[Lemma~1.3]{LS:1978}.
\end{proof}

In the Navier-Stokes-Voigt setting, note that, in view of (\ref{eq:uniq}) below, assumption (\ref{u:C(0,T;W1-infty)}) can be shown to be satisfied if one assume that, in addition to (\ref{hyp:gra:ro}), (\ref{eq:th1})-(\ref{eq:th11}) and (\ref{main:a:uniq}) are also verified.
By assuming these hypotheses are satisfied altogether, one can show the boundedness of $\|\nabla\rho(t)\|_{L^\infty(\Omega)}$ and $\|\rho_t(t)\|_{L^\infty(\Omega)}$, which in a certain sense can replace (\ref{eq:uniqs1}) and (\ref{eq:uniqs2}) in the Navier-Stokes-Voigt case
(see (\ref{eq:uniq+}) below).

The next result shows us how a higher regularity of the solutions depends on the smoothness of the problem data $\rho_0$, $u_0$ and $f$.

\begin{proposition}\label{th:A1}
Let $(u,\rho,p)$ be a solution of the problem (\ref{NSV:inc})-(\ref{NSV:u:bc}) in the conditions of Theorem~\ref{thm:e:strong:2}.
If, in addition to (\ref{hyp:gra:ro}), is verified (\ref{eq:th1})-(\ref{eq:th11}) and (\ref{main:a:uniq}),
then there exist positive constants $K_9$ and $K_{10}$ such that
\begin{alignat}{2}
&
\label{eq:uniq}
\sup_{t\in\left[0,T\right]}\left(\left\Vert D^2u(t)\right\Vert_{L^r(\Omega)}^2 +
\left\Vert\nabla u(t)\right\Vert_{L^\infty(\Omega)}^2\right)+
\left\Vert\nabla p\right\Vert_{L^{2}\left(0,T;L^{r}(\Omega)\right)}^2 \leq K_9, && \\
&
\label{eq:uniq+}
\sup_{t\in[0,T]}\left(\|\nabla \rho\|_{L^\infty(\Omega)}^2+\|\rho_{t}\|_{L^\infty(\Omega)}^2\right)\leq K_{10}. &&
\end{alignat}
\end{proposition}

\begin{proof}
Using the assumptions (\ref{hyp:gra:ro}) and (\ref{eq:th1})-(\ref{eq:th11}), we have proved in \cite[Theorem~4]{AOKh:2021:Nonl} that
\begin{alignat}{2}
&
\label{eq:uniq:0}
\sup_{t\in\left[0,T\right]}\left(\left\Vert {\Delta  }u(t)\right\Vert_{L^r(\Omega)}^2 +
\left\Vert\nabla u(t)\right\Vert_{C^{0,\alpha}(\overline{\Omega})}^2\right)+
\left\Vert\nabla p\right\Vert_{L^{2}\left(0,T;L^{r}(\Omega)\right)}^2 \leq C_1, && \\
&
\label{eq:uniq+:0}
\sup_{t\in[0,T]}\left(\|\nabla \rho\|_{L^\infty(\Omega)}^2+\|\rho_{t}\|_{L^\infty(\Omega)}^2\right)\leq C_2 &&
\end{alignat}
for some positive constants
$C_1=C\left(\mu,\kappa,M^\ast,d,\Omega,T,\|\nabla u_{0}\|_{L^2(\Omega)},\|\Delta u_{0}\|_{L^r(\Omega)},
\|f\| _{L^{2}(0,T;L^{r}(\Omega))}\right)$
and
$C_2=C\left(\mu,\kappa,M^\ast,d,\Omega,T,\|\nabla \rho_0\|^2_{L^\infty(\Omega)},\|\nabla u_0\|_{L^2(\Omega)},
\|\Delta u_0\|_{L^2(\Omega)},\|f\| _{L^{2}(0,T;L^{r}(\Omega))}\right)$.
The restriction (\ref{main:a:uniq}) results by the application of the Sobolev inequalities (\ref{Sob:ineq:gen}) and (\ref{Sob:ineq:2:D2}) as follows,
\begin{alignat*}{2}
&
\left\|u_t(t)\right\|_{L^{r}(\Omega)}\leq C(r,d,\Omega)\left\|\nabla u_t(t)\right\|_{L^{2}(\Omega)},\quad r\leq 2^\ast, && \\
&
\left\|\nabla u(t)\right\|_{L^{\infty}(\Omega)}\leq C(r,d)\left\|D^2 u(t)\right\|_{L^{r}(\Omega)},\quad r>d. &&
\end{alignat*}
It is immediate now that (\ref{eq:uniq}) and (\ref{eq:uniq+}) is a consequence of the estimates (\ref{eq:uniq:0})-(\ref{eq:uniq+:0}) by the application of (\ref{Sob:ineq:Lap:D2}).
\end{proof}

We are now in conditions to prove Theorem~\ref{thm:u:strong}.
In~\cite[Theorems~4-5]{{AOKh:2021:Nonl}} we already have proved a uniqueness result in the case of a strictly positive initial density and under the same assumptions (\ref{hyp:gra:ro}), (\ref{eq:th1})-(\ref{eq:th11}) and (\ref{main:a:uniq}).
The issue of wether or not the initial density may vanish in a subdomain of $\Omega$ does not matter for this result.
Therefore, we can prove Theorem~\ref{thm:u:strong} telegraphically, addressing the details to \cite{{AOKh:2021:Nonl}}.

\begin{proof} (Theorem~\ref{thm:u:strong})
Let $(\hat{u},\hat{p},\hat{\rho})$ and $(\overline{u},\overline{p},\overline{\rho})$  be two solutions of the problem (\ref{NSV:inc})-(\ref{NSV:u:bc}) with the same data.
By algebraic manipulations of the equations (\ref{NSV:inc})-(\ref{NSV:u:bc}), satisfied by each of these two couple of solutions, one has
\begin{alignat}{2}
& \label{uniq-3}%
\hat{\rho} u_t + \hat{\rho}\left(\hat{u}\cdot\nabla\right)u - \mu\Delta u- \kappa\Delta u_t+\nabla p =
\rho \left[f-\overline{u}_t-\left(\overline{u}\cdot\nabla\right)\overline{u}\right]-\hat{\rho}\left(u\cdot\nabla\right)\overline{u}
&& \\
& \label{uniq-5}
\rho_t + \nabla{\hat{\rho}}\cdot u + \nabla\rho\cdot\hat{u} = 0, && \\
& \label{eq:diver}%
\operatorname{div} u=0, &&
\end{alignat}
where $u=\hat{u}-\overline{u}$, $p=\hat{p}-\overline{p}$ and $\rho=\hat{\rho}-\overline{\rho}$.
Multiplying \eqref{uniq-3} by $u$, next integrating the resulting identity over $\Omega$ and using \eqref{eq:diver}, we have
\begin{equation}\label{uniq-4}
\begin{split}
&
\frac{1}{2}\frac{d}{dt}\left(\left\Vert\sqrt{\hat{\rho}(t)} u(t)\right\Vert^2_{L^2(\Omega)}
+
\kappa \left\Vert\nabla u(t) \right\Vert^2_{L^2(\Omega)}\right)
+
\mu\left\Vert\nabla u(t) \right\Vert^2_{L^2(\Omega)} = \\
&
\int_{\Omega}\Big(\rho(t)\left[f(t)-\overline{u}_t(t)-\left(\overline{u}(t)\cdot\nabla\right)\overline{u}(t)\right]
-
\hat{\rho}(t)\left(u(t)\cdot\nabla\right)\overline{u}(t)\Big)\cdot u(t) dx
\end{split}
\end{equation}
Multiplying now (\ref{uniq-5}) by $\rho$, integrating the resulting equation over $\Omega$, and using (\ref{eq:diver}), we also have
\begin{equation}\label{uniq-6}
\begin{split}
\frac{1}{2}\frac{d}{dt}\|\rho(t)\|^2_{L^2(\Omega)}=
-\int_{\Omega}\rho(t)(u(t)\cdot \nabla)\overline{\rho}(t)dx.
\end{split}\end{equation}
Adding up (\ref{uniq-4}) and (\ref{uniq-6}), we have
\begin{equation}\label{uniq-7}
\begin{split}
&\frac{1}{2}\frac{d}{dt}\left(\left\Vert \sqrt{\hat{\rho}(t)}u(t)\right\Vert^2_{L^2(\Omega)}+\kappa\left\Vert \nabla u(t)\right\Vert^2_{L^2(\Omega)}+\|\rho(t)\|^2_{L^2(\Omega)}\right)+\mu\left\Vert \nabla u(t)\right\Vert^2_{L^2(\Omega)}=\\
&
\int_{\Omega}\Big(\rho(t)\left[f(t)-\overline{u}_t(t)-\left(\overline{u}(t)\cdot\nabla\right)\overline{u}(t)\right]
-
\hat{\rho}(t)\left(u(t)\cdot\nabla\right)\overline{u}(t)\Big)\cdot u(t) dx
- \int_{\Omega}\rho(t)(u(t)\cdot \nabla)\overline{\rho}(t)dx
\end{split}\end{equation}
for all $t\in[0,T]$.
To estimate the r.h.s. terms of (\ref{uniq-7}), we proceed as we did in the proof of \cite[Theorem~5]{AOKh:2021:Nonl}, using in particular
(\ref{hyp:gra:ro}), (\ref{est:rho}) and (\ref{eq:uniq})-(\ref{eq:uniq+}).
Proceeding so, we obtain
\begin{equation}\label{uniq-8}
\begin{split}
&
\frac{d}{dt}\left(\left\Vert \sqrt{\hat{\rho}(t)}u(t)\right\Vert^2_{L^2(\Omega)}+\kappa\left\Vert \nabla u(t)\right\Vert^2_{L^2(\Omega)}+\|\rho(t)\|^2_{L^2(\Omega)}\right) \leq \\
&
A(t)\left(\left\Vert \sqrt{\hat{\rho}(t)}u(t)\right\Vert^2_{L^2(\Omega)}+\kappa\left\Vert \nabla u(t)\right\Vert^2_{L^2(\Omega)}+\|\rho(t)\|^2_{L^2(\Omega)}\right)
\end{split}
\end{equation}
for all $t\in[0,T]$ and for some function $A$ that, in view of Theorem~\ref{thm:e:strong:2}, (\ref{hyp:gra:ro}), (\ref{est:rho}), and (\ref{eq:uniq})-(\ref{eq:uniq+}), can be proven to belong to $L^1(0,T)$.
Therefore we can apply the Grönwall inequality to (\ref{uniq-8}), which will finally allow us to conclude that $\hat{\rho}=\overline{\rho}$ and $\hat{u}=\overline{u}$.
\end{proof}

 \section*{Acknowledgments}
All authors were supported by the Grants no. AP19676624 of
the Ministry of Education and Science of the Republic of Kazakhstan (Kazakhstan).
The first author was also partially supported by the Portuguese Foundation for Science and Technology (FCT) under the project no. UIDB/04561/2020 (Portugal).

\end{document}